\documentclass[12pt]{amsart}
\usepackage{geometry}
\usepackage{graphicx} 
\usepackage{adjustbox}
\usepackage{amsmath,amssymb,amsxtra,tikz,stackengine,bm,mathrsfs,bbm,tikz-cd,comment,hyperref,combelow,amsopn,float}
\usepackage{enumitem}
\usepackage{ytableau}
\usepackage{quiver}
\usetikzlibrary{calc}%
\usetikzlibrary{shapes}%
\usetikzlibrary{patterns}%
\usetikzlibrary{positioning}%
\usetikzlibrary{arrows.meta}
\usetikzlibrary{knots}
\usetikzlibrary{decorations.markings}
\usetikzlibrary{hobby}

\newtheorem{theorem}{Theorem}[section]
\newtheorem{proposition}[theorem]{Proposition}
\newtheorem{corollary}[theorem]{Corollary}
\newtheorem{lemma}[theorem]{Lemma}

\theoremstyle{definition}
\newtheorem{remark}[theorem]{Remark}
\newtheorem{definition}[theorem]{Definition}
\newtheorem{example}[theorem]{Example}

\numberwithin{equation}{section}

\newcommand{\bigslant}[2]{{\raisebox{.2em}{$#1$}\left/\raisebox{-.2em}{$#2$}\right.}}

\newcommand{\manifold}{\mathcal{M}}
\newcommand{\deep}{\mathcal{D}}

\newcommand{\PP}{\mathbb{P}}

\newcommand{\Z}{\mathbb{Z}}

\newcommand{\seeds}{\mathsf{Seeds}}
\newcommand{\clusters}{\mathsf{clusters}}
\definecolor{babyblueeyes}{rgb}{0.63, 0.79, 0.95}

\newcommand{\F}{\mathbb{F}}
\newcommand{\K}{\mathbb{K}}

\def\seed{\Sigma}

\def\x2{x^{(2)}}

\def\u1{u^{(1)}}
\def\w0{\Delta}

\def\cluster{\mathbf{x}}
\DeclareMathOperator{\Spec}{Spec}
\def\var{\mathcal{V}}
\def\zero{\mathbf{0}}
\def\one{\mathbf{1}}

\newcommand{\T}{\mathbb{T}}
\newcommand{\function}{\Upsilon}

\def\A{\mathsf{A}}
\def\D{\mathsf{D}}
\def\E{\mathsf{E}}

\def\x{\textbf{x}}

\definecolor{aquamarine}{rgb}{0.5, 1.0, 0.83}
\definecolor{aqua}{rgb}{0.0, 1.0, 1.0}


\title[Clusters, hexagonal moves, and coverings]{Cluster tori over $\mathbb{F}_2$, hexagonal moves on triangulations, and minimal coverings of cluster manifolds}
\author{Daniel P\'erez Melesio}
\address{Facultad de Ciencias, Universidad Nacional Aut\'onoma de M\'exico. Ciudad Universitaria, CDMX, M\'exico;
\newline
{\tiny{Current address:}} Instituto de Matem\'aticas, Universidad Nacional Aut\'onoma de M\'exico. Ciudad Universitaria, CDMX, M\'exico}
\email{perezmelesiod@ciencias.unam.mx}

\author{Jos\'e Simental}
\address{Instituto de Matem\'aticas, Universidad Nacional Aut\'onoma de M\'exico. Ciudad Universitaria, CDMX, M\'exico}
\email{simental@im.unam.mx}

\begin{document}

\begin{abstract}
In this paper, we study cluster algebras over $\F_2$. By the Laurent phenomenon there is a map from the set of seeds of the cluster algebra to the corresponding cluster variety. We show that in type $A$, fibers of this map can be described in terms of certain edges of the universal polytope of triangulations of a polygon. Moreover, we show that there is a section of this map giving seeds whose corresponding cluster tori cover the cluster manifold over any field $\F$, but there are also sections giving seeds whose cluster tori do not cover the cluster manifold over any field $\F \not\cong \F_2$. 
\end{abstract}

\maketitle

\section{Introduction}

\subsection{Motivation} In this paper, we study cluster algebras over the field with two elements $\F_2$.  The distinguishing feature of this field is that a torus over $\F_2$ is set-theoretically a point, so by the Laurent phenomenon we naturally get a map
\begin{equation}\label{eq:seeds-to-points}
\mathsf{Seeds}(A) \to \Spec_{\F_2}(A),
\end{equation}
where, if $A$ is an $\F$-algebra, we denote by $\Spec_{\F}(A)$ the variety of all (unital) ring homomorphisms $A \to \F$. Outside of very few cases, the map \eqref{eq:seeds-to-points} is not injective: indeed, if the cluster algebra $A$ is Noetherian then the right-hand side of \eqref{eq:seeds-to-points} is finite, while the left-hand side is infinite unless $A$ is of finite cluster type. But even in finite-cluster type, the map \eqref{eq:seeds-to-points} is not injective outside of the cases $\A_1, \A_2$ and $\A_3$. 

Our motivation for studying the map \eqref{eq:seeds-to-points} comes from the study of the \emph{deep locus} of cluster algebras. Let $Q$ be an ice quiver, and denote by $A(Q)$ the cluster algebra of $Q$ defined over $\Z$. We will assume that $A(Q)$ is locally acyclic, cf. \cite{Muller-locally-acyclic}. If $\F$ is a field, we denote by $A_{\F}(Q) := \F \otimes_{\Z} A(Q)$. By \cite[Lemma 4.7]{BMRS}, $A_\F(Q)$ coincides with the cluster algebra of $Q$ defined over $\F$.  The \emph{cluster variety} is $\var_\F(A) := \Spec_\F(A_\F(Q))$. By the Laurent phenomenon, every cluster $\cluster$ of $A(Q)$ defines a Zariski open set $\T_{\cluster}^{\F} \cong (\F^{\times})^{\#Q_0} \subseteq \var_\F(A)$. The \emph{cluster manifold} is $\manifold_\F(A) := \bigcup_{\x} \T_\x^{\F} \subseteq \var_\F(A)$. This is a Zariski open subvariety of $\var_\F(A)$. The \emph{deep locus} if $\deep_\F(A) := \var_\F(A) \setminus \manifold_\F(A)$, that is, the complement to the union of cluster tori in $\var_\F(A)$. To put it more succinctly, a homomorphism $\varphi: A_\F(Q) \to \F$ belongs to $\deep_\F(A)$ if, for every cluster $\x$, there exists a cluster variable $x \in \x$ such that $\varphi(x) = 0$.  Note that the image of \eqref{eq:seeds-to-points} is always contained in $\manifold_{\F_2}(A)$, and in fact the map
\begin{equation}\label{eq:seeds-to-points-refined}
\seeds(A) \to \manifold_{\F_2}(A)
\end{equation}
is surjective. Note that $\manifold_\F(A)$, being a union of tori, is always a smooth $\F$-variety. So anything badly behaved on the cluster algebra can be blamed on its deep locus, cf. \cite{beyer-muller}. 

One of the difficulties in studying the deep locus is that, by definition, the cluster manifold $\manifold_\F(A)$ is typically the union of an \emph{infinite} number of cluster tori. Nevertheless, if the algebra $A$ is Noetherian then $\manifold_\F(A)$ is Noetherian with the Zariski topology, so a finite number of cluster tori suffice to cover $\manifold_\F(A)$. However, determining a specific finite set of clusters whose cluster tori cover $\manifold_\F(A)$ seems to be a challenging problem. 

We say that a set of seeds $\Omega \subseteq \seeds(A)$ is an \emph{$\F$-covering} if the union of its cluster tori (defined over $\F$) covers $\manifold_\F(A)$. As we will see, the property of being a covering depends on the field $\F$. In some cases (for example, cluster algebras coming from braid varieties \cite{CGGLSS, GLSBS, GLSB}, which include cluster algebras of finite cluster type) there is a natural set inclusion of $\Spec_{\F_2}(A)$ into $\Spec_\F(A)$ for any field $\F$, and an $\F$-covering set must contain an $\F_2$-covering set. Note that the $\F_2$-covering sets are precisely the (images of) sections of the map \eqref{eq:seeds-to-points-refined}. Thus, in this paper we study such sections, with specific focus on finite cluster type $\A$. In particular, using the usual bijection between seeds and triangulations of a convex $n$-gon we describe the fibers of \eqref{eq:seeds-to-points-refined} in terms of certain edges of the \emph{universal polytope} of triangulations of the $n$-gon \cite{BFS-universal, Jesus-universal}, that we call \emph{hexagonal flips}. Note that the usual flips, corresponding to cluster mutation, are also edges of the universal polytope (in fact, they are precisely the edges of Gelfand-Kapranov-Zelevinsky's \emph{secondary polytope}). It would be interesting to find whether other edges of the universal polytope have a cluster-theoretic interpretation. 

\subsection{Results} Our first result concerns point counts of \emph{acyclic} cluster varieties over $\F_2$. Recall that an ice quiver $Q$ is acyclic if $Q$ has no directed cycles consisting only of mutable vertices. In particular, the mutable part of such a quiver must have a sink.

\begin{lemma}\label{lem:point-count-intro}
Let $Q$ be an acyclic ice quiver, and let $x$ be a mutable vertex that is a sink of the mutable part of $Q$. Let $Q^{-x}$ be the quiver obtained from $Q$ by deleting $x$, and let $Q^{-N(x)}$ be the quiver obtained by deleting $x$ and all vertices adjacent to it. Then,
\begin{equation}\label{eq: point count recursion}
\#\var_{\F_2}(A(Q)) = \#\var_{\F_2}(A(Q^{-x}) + 2\#\var_{\F_2}(A(Q^{-N(x)})).
\end{equation}
\end{lemma}

Note that \eqref{eq: point count recursion} does not depend on the number of frozens. Indeed, we make the convention that frozen variables are always invertible, so every point in the cluster variety has to evaluate to $1$ in each frozen vertex. See, however, Remark \ref{rmk:degenerate-case}. We also remark that \cite[Proposition 3.9]{LSII} gives a formula for the point count of an acyclic cluster variety of really full rank over a finite field in terms of the combinatorics of the anticliques of the quiver $Q$. 

With \eqref{eq: point count recursion} in hand, it is an easy exercise to compute the number of $\F_2$-points of cluster varieties of finite cluster type, see Table \ref{table:finite-cluster-type}. We remark that if $Q$ is of type $\A_1$, then $\#\var_{\F_2}(A(Q)) = 3$, see Remark \ref{rmk:small-cases}.

\begin{table}
\begin{center}
\begin{tabular}{|c|c|c|c|}
\hline
Cluster type & $\#\var_{\F_2}$ & $\#\deep_{\F_2}$  & $\#\manifold_{\F_2}$\\ [0.25em]
\hline
$\A_{2n}$ & $\frac{2^{2n+2} - 1}{3}$ & 0 & $\frac{2^{2n+2} -1}{3}$   \\[0.25em]
\hline
$\A_{2n+1}$ & $\frac{2^{2n+3} + 1}{3}$ & 1 & $\frac{2^{2n+3} -2}{3}$ \\[0.25em] 
\hline
$\D_{2n}$ $(n \geq 2)$ & $\frac{5(2^{2n}) + 7}{3}$ & $\frac{2^{2n-1} + 13}{3}$  & $3(2^{2n-1}) - 2$ \\ [0.25em]
\hline
$\D_{2n+1}$ $(n \geq 2)$ & $\frac{5(2^{2n+1}) - 7}{3}$ & $\frac{2^{2n} - 1}{3}$  & $3(2^{2n}) - 2$ \\ [0.25em] \hline
$\E_6$ & $93$ & $0$ & $93$  \\ [0.25em]
\hline
$\E_7$ & $195$ & $5$  & $190$  \\ [0.25em]
\hline
$\E_8$ & $381$ & $0$ & $381$ \\ [0.25em]
\hline 
\end{tabular}
\end{center}
\caption{Number of points of the cluster variety, deep locus, and cluster manifold, of simply-laced finite-type cluster varieties over $\F_2$. The second column is obtained from Lemma \ref{lem:recursion} below, and the third column uses \cite[Theorem 6.8]{deep-locus}: the deep locus of a cluster variety of type $\E_7$ is a cluster variety of type $\A_2$; the deep locus of a cluster variety of type $\D_{2n+1}$ is a cluster variety of type $\A_{2n-2}$; and the deep locus of a cluster variety of type $\D_{2n}$ is the union of a cluster variety of type $\A_{2n-3}$ and two cluster varieties of type $\A_{1}$, so that the intersection of any two of them is a single point.}
\label{table:finite-cluster-type}
\end{table}

As soon as $n > 3$, the number of $\F_2$-points of a cluster variety of type $\A_n$ is smaller than the number of seeds of the corresponding type. We recall that seeds of cluster type $\A_n$ are in bijection with triangulations of a regular $(n+2)$-gon. 

\begin{theorem}\label{thm:main-intro}
Let $\seed_1, \seed_2$ be two seeds of a cluster algebra of type $\A_n$, corresponding to the triangulations $T_1$ and $T_2$, respectively. Then, $\seed_1$ and $\seed_2$ determine the same point in $\manifold_{\F_2}(\A_n)$ if and only if $T_2$ can be obtained from $T_1$ applying a sequence of the local moves on triangulatons from Figure \ref{fig:hexagonal-moves-intro}.
\end{theorem}

\begin{figure}
    \begin{center}
\begin{tikzpicture}
   \newdimen\R
\R=1cm
   \draw (0:\R)
   \foreach \x in {60,120,...,360} {  -- (\x:\R) };
   \draw (60:\R)--(180:\R)--(360:\R)--(240:\R);
\draw[<->] (1.5,0) -- (2.5, 0);
\begin{scope}[shift={(4,0)}]
   \draw (0:\R)
   \foreach \x in {60,120,...,360} {  -- (\x:\R) };
   \draw (60:\R)--(300:\R)--(120:\R)--(240:\R);
\end{scope}
\begin{scope}[shift={(7,0)}]
   \draw (0:\R)
   \foreach \x in {60,120,...,360} {  -- (\x:\R) };
   \draw (60:\R)--(180:\R)--(300:\R)--(60:\R);
\end{scope}
\draw[<->] (8.5,0) -- (9.5,0);
\begin{scope}[shift={(11,0)}]
   \draw (0:\R)
   \foreach \x in {60,120,...,360} {  -- (\x:\R) };
   \draw (360:\R)--(120:\R)--(240:\R)--(360:\R);
\end{scope}
\end{tikzpicture}
\end{center}
\caption{Local hexagonal moves on triangulations. The figure on the left-hand side relates the two \lq zig-zag\rq \, triangulations joining the same pair of antipodal points, and the figure on the right-hand side relates the two inscribed triangles in the hexagon.}
\label{fig:hexagonal-moves-intro}
\end{figure}
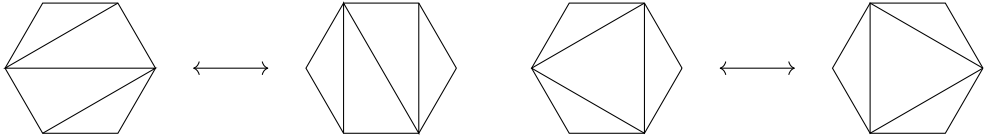

Note that, from Table \ref{table:finite-cluster-type} we see that an $\F_2$-covering set for a cluster variety of type $\A_n$ must have at least
\begin{equation}\label{eq:minimal-number-covering-set}
\begin{cases} \frac{2^{n+2}+(-1)^{n+1}}{3} \; \text{elements} & \text{if $n$ is even}, \\
\frac{2^{n+2} + (-1)^{n+1}}{3} - 1 \; \text{elements} & \text{if $n$ is odd.}\end{cases}
\end{equation}

The next result shows that, for any field, we can find a covering set with exactly \eqref{eq:minimal-number-covering-set} elements. However, there are $\F_2$-covering sets that are not covering sets for any other field $\F$. 

\begin{theorem}\label{thm:covering-intro}
Let $n>0$, and consider a cluster algebra of type $\A_n$. Then
\begin{enumerate}
\item[(a)] There exists a minimal $\F_2$-covering set that is also a minimal $\F$-covering set for any field $\F$.
\item[(b)] In type $\A_{9}$, there exists an $\F_2$-covering set that is not an $\F$-covering set for any field $\F \not\cong \F_2$. 
\end{enumerate}
\end{theorem}

We believe that, in fact, for any $n \geq 9$ there exists an $\F_2$-covering set that is not an $\F$-covering set for any field $\F \not\cong \F_2$, but we do not show this.

\subsection{Structure of the paper} In Section \ref{sec:cluster-algebras} we recall general results on acyclic cluster algebras and prove Lemma \ref{lem:point-count-intro} as Lemma \ref{lem:recursion}. In Section \ref{sec:type-A} we focus on type $\A$ cluster algebras. First, we give a geometric model for the corresponding cluster variety that works over any field, and examine the cluster tori combinatorially. In this section, we prove Theorem \ref{thm:main-intro} as Theorem \ref{thm:main-text}. Finally, in Section \ref{sec:minimal-coverings} we study minimal covering sets in type $\A$, we prove Theorem \ref{thm:covering-intro}(a) as Proposition \ref{prop:covering-1}, and in Section \ref{sec:counterexample} we prove Theorem \ref{thm:covering-intro}(b). 

\subsection*{Acknowledgments} This work is part of the first author's bachelor's thesis at Facultad de Ciencias of UNAM under the supervision of the second author. We are grateful to Jesús de Loera for useful conversations, and Melissa Sherman-Bennett for useful conversations and comments on am earlier draft of this work. We also thank the anonymous referee for a careful reading and suggestions that improved the readability of the paper. Our work was partially supported by SECIHTI Project CF-2023-G-106 and UNAM'S PAPIIT Grant IA101526.

\section{Cluster algebras}\label{sec:cluster-algebras}

In this section, we succinctly recall the definition of cluster algebras as well as some preliminary results on the theory. 

\subsection{Definition} Let $R$ be an integral domain, of arbitrary characteristic. The main combinatorial input defining a cluster algebra (over $R$) is that of a \emph{seed}. Consider a purely transcendental extension $R \subseteq \K$ of transcendence degree $n+m$, i.e. $\K \cong R(y_1, \dots, y_{n+m})$ for algebraically independent (over $R$) elements $y_1, \dots, y_{n+m}$. A (rank $n+m$) extended seed $\seed = (\x, Q)$ consists of 
\begin{enumerate}
\item A set $\x = \{x_1, \dots, x_n, x_{n+1}, \dots, x_{n+m}\} \subseteq \K$ such that $\K = R(x_1, \dots, x_{n+m})$. This set is called the \emph{cluster} of the seed $\seed$.
\item A quiver $Q$ with $n+m$ vertices, labeled by $1, \dots, n+m$. The vertices $1, \dots, n$ are called \emph{mutable}, and the vertices $n+1, \dots, n+m$ are \emph{frozen}. We will assume that $Q$ has no loops or directed $2$-cycles and that there are no arrows between frozen vertices. 
\end{enumerate}

Given a mutable vertex $k \leq n$, the process of \emph{mutation} (see e.g. \cite[Chapter 2]{FWZbook1}) produces another seed $\mu_k(\seed) = (\mu_k(\x), \mu_k(Q))$. The cluster $\mu_k(\x)$ is of the form $\mu_k(\x) = \x \setminus \{x_k\} \cup \{x'_k\}$, where $x'_k$ is defined by the \emph{exchange relation}
\begin{equation}\label{eq:mutation}
x'_kx_k = \prod_{\substack{\alpha \in Q_1 \\ s(\alpha) = k}}x_{t(\alpha)} + \prod_{\substack{\beta \in Q_1 \\ t(\beta) = k}} x_{s(\beta)}.
\end{equation}

Mutation is involutive in that $\mu_k(\mu_k(\seed)) = \seed$. Thus, the following is an equivalence relation among seeds:
\[
\seed \sim \seed' \qquad \text{if} \; \seed \; \text{can be reached from} \; \seed' \; \text{by an iterated sequence of mutations.}
\]

Note that if $\seed' = (\x', Q')$ then $x_{n+1}, \dots, x_{n+m} \in \x'$. 

\begin{definition}
The \emph{cluster algebra} $A_R(\seed)$ is the $R[x_{n+1}^{\pm1}, \dots, x_{n+m}^{\pm1}]$-subalgebra of $\K$ generated by the set
\[
\mathbf{X} = \bigcup_{\seed \sim \seed' = (\x', Q')} \x'.
\]
\end{definition}

Note that any seed $\seed'$ mutation equivalent to $\seed$ gives rise to an isomorphic algebra, $A_R(\seed) \cong A_R(\seed')$. Given a cluster algebra $A = A_R(\seed)$ we will denote by $\seeds(A)$ the set of all seeds that are mutation equivalent to $\seed$. Similarly, we denote by $\clusters(A)$ the set of all clusters of seeds that are mutation equivalent to $\seed$. 

\begin{remark}\label{rmk:degenerate-case}
As in \cite{BMRS}, we will make the assumption that if $Q$ contains an isolated vertex then this vertex must be frozen. This assumption is made to deal with the case when the characteristic of $R$ is $2$. Indeed, if $k$ is an isolated mutable vertex, then the exchange relation \eqref{eq:mutation} becomes $x_kx'_k = 1 + 1 = 2$, so if $\mathrm{char}(R) = 2$ then one of the variables $x_k, x_k'$ must be zero and mutation ceases to be uniquely defined. 
\end{remark}

\subsection{Cluster tori and the deep locus} 
We now assume that the base ring for a cluster algebra is a field $\F$. We denote by
\[
\var_\F(\seed) := \Spec_{\F}(A(\seed)),
\]
i.e., the variety of all unital ring homomorphisms $A(\seed) \to \F$. Now let $\x' \in \clusters(A)$. The \emph{Laurent phenomenon}, cf. \cite{FZ_laurent}  asserts that
\begin{equation}\label{eq:laurent}
A[(x'_1)^{-1}, \dots, (x'_n)^{-1}, x_{n+1}^{-1}, \dots, x_{n+m}^{-1}] = \F[(x'_1)^{\pm 1}, \dots, (x'_n)^{\pm 1}, x_{n+1}^{\pm 1}, \dots, x_{n+m}^{\pm 1}]
\end{equation}

Geometrically, the Laurent phenomenon asserts that the (Zariski open) locus of points $\varphi \in \Spec_\F(A)$ satisfying $\varphi(x'_i) \neq 0$ for all $i = 1, \dots, n$ is isomorphic to a torus $(\F^{\times})^{n+m}$. We call this locus a \emph{cluster torus}, and denote it by $\T_{\x'}$. Moreover, we call the set
\[
\manifold_\F(\seed) := \bigcup_{\x \in \clusters(A)} \T_{\x} \subseteq \var_\F(\seed).
\]
the \emph{cluster manifold} of $A$. The complement $\deep_\F(\seed) := \var_\F(\seed) \setminus \manifold(\seed)$ is called the \emph{deep locus} of $A(\seed)$. 

\begin{remark}\label{rmk:frozens-dont-matter-f2}
    Note that in this paper we take the convention that frozen variables are invertible in $A(\seed)$. In particular, if $\varphi \in \var_\F(\seed)$ and $x_{i}$ is a frozen variable, then $\varphi(x_i) \neq 0$. If, moreover, $\F = \F_2$, this forces $\varphi(x_i) = 1$. Note that this implies that, if $\seed$ and $\seed'$ are two seeds whose mutable parts are equal but which may differ in their frozen part, then
    \[
    \var_{\F_2}(\seed) = \var_{\F_2}(\seed').
    \]
     For example, when talking about a quiver of type $\A_1$, we will assume, keeping in line with Remark \ref{rmk:degenerate-case}, that the associated quiver is $\circ \to \square$. 
\end{remark}

\subsection{The acylic case} We now assume that the seed $\seed$ is such that the quiver $Q$ is acyclic, that is, there are no directed cycles involving only mutable vertices. In this case, we have (cf. \cite[Corollary 1.21]{BFZ})
\[
A(\seed) \cong \bigslant{\F[x_1, x'_1 \dots, x_{n}, x'_n, x_{n+1}^{\pm 1}, \dots, x_{n+m}^{\pm 1}]}{\left(x_kx'_k - \prod_{\substack{\alpha \in Q_1 \\ s(\alpha) = k}}x_{t(\alpha)} - \prod_{\substack{\beta \in Q_1 \\ t(\beta) = k}} x_{s(\beta)}\right)}
\]
where $k$ in the generating set of the ideal runs from $1$ to $n$. 
Thus, an element $\varphi \in \var_{\F}(\seed)$ is simply a collection of $2n + m$ elements 
\[\varphi(x_1), \dots, \varphi(x_{n+m}), \varphi(x'_1), \dots, \varphi(x'_n) \in \F
\]
such that $\varphi(x_i)$ is invertible for $i > n$ and $\varphi(x_i)\varphi(x'_i)$ satisfies the exchange relation \eqref{eq:mutation}. This has the following consequence.

\begin{lemma}\label{lem:recursion}
Let $\seed = (\cluster, Q)$ be an acyclic seed, and let $i \in Q_0$ be a mutable sink or source. Let $\seed^{-i}$ be seed obtained from $\seed$ by deleting the vertex $i$ from $Q$ and the variable $x_i$ from $\cluster$, and let $\seed^{-N(i)}$ be the seed whose quiver is obtained from $Q$ by deleting the vertex $i$ and all of its neighbors, and its cluster is obtained by deleting the corresponding variables from $\cluster$. Then,
\[
\#\var_{\F_2}(\seed) = \#\var_{\F_2}(\seed^{-i}) + 2\#\var_{\F_2}(\seed^{-N(i)}).
\]
\end{lemma}
\begin{proof}
Let $\varphi \in \var_{\F_2}(\seed)$. We examine the possible values of $\varphi(x_i)$. 

If $\varphi(x_i) = 1$, then the restriction of $\varphi$ to all variables which are not $x_i, x'_i$ defines an element of $\var_{\F_2}(\seed^{-i})$. Moreover, an element of $\tilde{\varphi} \in \var_{\F_2}(\seed^{-i})$ defines an element $\varphi \in \var_{\F_2}(\seed)$ by defining $\varphi(x_i) = 1$ and $\varphi(x'_i) = \prod_{\substack{\alpha \in Q_1 \\ s(\alpha) = i}}\tilde{\varphi}(x_{t(\alpha)}) + \prod_{\substack{\beta \in Q_1 \\ t(\beta) = i}}\tilde{\varphi}(x_{t(\beta)})$. Thus, the set of elements $\varphi \in \var_{\F_2}(\seed)$ such that $\varphi(x_i) = 1$ are in bijection with $\var_{\F_2}(\seed^{-i})$. 

If $\varphi(x_i) = 0$ then we must have (say, in the case where $i$ is a sink)
\[
0 = 1 + \prod_{\substack{\beta \in Q_1 \\ t(\beta) = i}}\varphi(x_{s(\beta)})
\]
so that $\varphi(x_{k}) = 1$ for every $k$ that has an arrow $k \to i$. Thus, $\varphi$ defines an element of $\var_{\F_2}(\seed^{-N(i)})$. On the other hand, if $\tilde{\varphi} \in \var_{\F_2}(\seed^{-N(i)})$, then $\tilde{\varphi}$ defines two distinct elements $\varphi_1, \varphi_2 \in \var_{\F_2}(\seed)$ by declaring $\varphi_{1,2}(x_k) = 1$ for every $k$ that has an arrow $k \to i$, $\varphi_{1,2}(x_i) = 0$ (note that this specifies uniquely $\varphi_{1,2}(x'_k)$ if there is an arrow $k \to i$) and $\varphi_1(x'_i) = 1$, $\varphi_2(x'_i) = 0$. Thus, the elements $\varphi \in \var_{\F_2}(\seed)$ with $\varphi(x_i) = 0$ are in bijection with $\var_{\F_2}(\seed^{-N(i)}) \sqcup \var_{\F_2}(\seed^{-N(i)})$. This finishes the proof. 
\end{proof}

\begin{remark}\label{rmk:small-cases}Note that if $\seed$ is the empty seed, then $A_{\F_2}(\seed) = \F_2$, so $\#\var_{\F_2}(\seed) = 1$. If $\seed$ has a unique mutable vertex, then $\var_{\F_2}(\seed) \cong \{(x,x') \in \F_2^{2} \mid xx' = 0\} = \F_2^{2} \setminus \{(1,1)\}$, so $\#\var_{\F_2}(\seed) = 3$. 
From here, one can  recursively determine the number of elements in $\var_{\F_2}(\seed)$ for any acyclic seed $\seed$.
\end{remark}

\section{Type \texorpdfstring{$\A$}{A}}\label{sec:type-A}

For the remainder of the paper, we focus on the case of type $\A$ cluster varieties. 

\subsection{A geometric model} It will be convenient to have a geometric model for type $\A$ cluster varieties. Over $\F_2$, we can focus on this geometric model without loss of generality by Remark \ref{rmk:frozens-dont-matter-f2}. We consider the projective space $\F\PP^1 = \{[a:b] \mid (a,b) \in \F^2, (a,b) \neq (0,0)\}$, and the following elements of $\F\PP^1$:
\[
\zero := [1:0], \qquad \infty := [0:1].
\]

And consider the set $X_{\F}(m)$ consisting of $(m+1)$-tuples $(y_0, \dots y_{m+1}) \in (\F\PP^1)^{m+1}$ satisfying the following conditions:
\[
 y_0 = \zero, y_m = \infty, y_i \neq y_{i+1} \; \text{for all} \; i = 0, \dots, m-1
\]

Note that $X_\F(m)$ is an affine algebraic variety: given an element $(y_0, \dots, y_m) \in X(m)$ there exists a unique collection $a_i, b_i \in \F$, $i = 0, \dots, m$ such that
\begin{enumerate}
\item[(1)] $(a_0, b_0) = (1,0), (a_m, b_m) = (0,1)$. 
\item[(2)] $y_i = [a_i:b_i]$ for every $i = 0, \dots, m$.
\item[(3)] $\det\begin{pmatrix} a_i & a_{i+1} \\ b_i & b_{i+1}\end{pmatrix} = 1$ for $i = 0, \dots, m-2$.
\item[(4)] $\det\begin{pmatrix} a_{m-1} & a_m \\ b_{m-1} & b_m\end{pmatrix} = \det\begin{pmatrix} a_{m-1} & 0 \\ b_{m-1} & 1 \end{pmatrix} = a_{m-1} \neq 0$.
\end{enumerate}

The equations in (3), together with $a_{m-1} \neq 0$, give the structure of an affine algebraic variety over $\F$ to $X_\F(m)$. 

We consider a regular polygon $P_{m+1}$ with vertices $0, 1, \dots, m$, with a distinguished side joining $m$ and $0$, that we always color in purple. We can write elements of $X_\F(m)$ as clockwise labelings of the vertices $P_{m+1}$, so that the element $(y_0, \dots, y_m)$ represents the labeling that labels the vertex $i$ by $y_i$. We can then give a cluster structure on $\F[X_\F(m)]$ as follows.
\begin{itemize}
\item Seeds are in bijection with triangulations of the polygon $P_{m+1}$.
\item For a triangulation $T$ of $P_{m+1}$, the mutable part of the cluster $\x$ is given by $\left\{\det\begin{pmatrix} a_i & a_j \\ b_i & b_j\end{pmatrix}\right\}$, where $ij$ runs over the diagonals in $T$. The unique frozen variable corresponds to the side $(m-1,m)$, and it is $\det\begin{pmatrix} a_{m-1} & 0 \\ b_{m-1} & 1\end{pmatrix} = a_{m-1}$. 
\item Given a triangulation $T$ of $P_{m+1}$, a quiver is obtained by placing a mutable vertex in each internal diagonal of $T$, a frozen variable on the side $(m-1, m)$, and arrows that create a counterclockwise $3$-cycle in each triangle.
\item Given a triangulation $T$ and an internal diagonal of it, it belongs to exactly two triangles that together form a quadrilateral. Mutation at the vertex associated to this diagonal corresponds to replacing it by the other diagonal in the quadrilateral, see Figure \ref{fig:mutation}. 
\end{itemize}

By choosing a \emph{fan triangulation} (i.e., a triangulation all whose diagonals are incident to one vertex) we can see that $\F[X_\F(m)]$ is a cluster algebra of type $\A_{m-2}$. 

Using the description of the cluster variables, it is easy to verify when an element $(y_0, \dots, y_m)$ belongs to a given cluster torus. 

\begin{lemma}\label{lem:cluster-torus-triangulation}
    Let $T$ be a triangulation of $P_{m+1}$, $\x$ its associated cluster, and $\T_\x$ the corresponding cluster torus. Then $(y_0, \dots, y_m) \in \T_\x$ if and only if $y_i \neq y_j$ for each diagonal $ij \in T$. 
\end{lemma}

\begin{figure}[ht!]
\begin{center}
\begin{tikzpicture}
   \newdimen\R
\R=1cm
   \draw (0:\R)
   \foreach \x in {60,120,...,360} {  -- (\x:\R) };
   \draw (60:\R)--(180:\R)--(360:\R)--(240:\R);
   \draw[color=purple] (240:\R)--(300:\R);
    \node at (240:1.2cm) {\tiny{$y_0$}};
    \node at (180:1.2cm) {\tiny{$y_1$}};
    \node at (120:1.2cm) {\tiny{$y_2$}};
    \node at (60:1.2cm) {\tiny{$y_3$}};
    \node at (360:1.2cm) {\tiny{$y_4$}};
    \node at (300:1.2cm) {\tiny{$y_5$}};

    \node at (0,0) {\color{red}\tiny{$\bullet$}};
    \node at (-0.2, 0.45) {\tiny{$\bullet$}};
    \node at (0.2, -0.45) {\tiny{$\bullet$}};
    \node at (0.75, -0.45) {\tiny{$\square$}};

    \draw[->] (0.65, -0.45) -- (0.3, -0.45);
    \draw[->] (0.2, -0.4) -- (0.05, -0.05);
    \draw[<-] (0, 0.05) -- (-0.15, 0.4);

\begin{scope}[shift={(4,0)}]
   \draw (0:\R)
   \foreach \x in {60,120,...,360} {  -- (\x:\R) };
   \draw (60:\R)--(180:\R);
   \draw (60:\R)--(240:\R) -- (360:\R);
   \draw[color=purple] (240:\R)--(300:\R);
    \node at (240:1.2cm) {\tiny{$y_0$}};
    \node at (180:1.2cm) {\tiny{$y_1$}};
    \node at (120:1.2cm) {\tiny{$y_2$}};
    \node at (60:1.2cm) {\tiny{$y_3$}};
    \node at (360:1.2cm) {\tiny{$y_4$}};
    \node at (300:1.2cm) {\tiny{$y_5$}};

    \node at (0,0) {\color{red}\tiny{$\bullet$}};
    \node at (-0.2, 0.45) {\tiny{$\bullet$}};
    \node at (0.2, -0.45) {\tiny{$\bullet$}};
    \node at (0.75, -0.45) {\tiny{$\square$}};

     \draw[->] (0.65, -0.45) -- (0.3, -0.45);
    \draw[<-] (0.2, -0.4) -- (0.05, -0.05);
    \draw[->] (0, 0.05) -- (-0.15, 0.4);
\end{scope}
\end{tikzpicture}
\end{center}
\caption{Two seeds of the variety $X_\F(5)$, that are related by mutation at the red vertex. Note that the cluster torus associated to the seed on the left-hand side is $\{y_0 \neq y_4\}\cap\{y_4 \neq y_1\}\cap\{y_1 \neq y_3\}$, while the torus associated to the seed on the right-hand side is $\{y_0 \neq y_4\}\cap\{y_0 \neq y_3\}\cap\{y_3 \neq y_1\}$.}
\label{fig:mutation}
\end{figure}
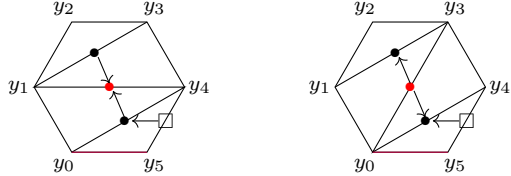

By \cite[Proposition 5.1]{deep-locus}, the union of all cluster tori consists of $X_\F(m)$ except the point $\{(\zero, \infty, \zero, \infty, \dots, \zero, \infty)\}$ when it belongs to $X_\F(m)$. In particular, $X(m)$ has empty deep locus if and only if $m+1$ is odd.

\begin{definition}\label{def:valid-diagonal-point}
    Let $y = (y_0, \dots, y_m) \in X(m)$, and let $ij$ be a diagonal of $P_{m+1}$. We say that a diagonal $ij$ is \emph{invalid} for $y$ if one of the two things happens.
    \begin{enumerate}
    \item $y_i = y_j$, or
    \item $ij$ separates $P_{m+1}$ into two polygons, and in one of these polygons the labels of the vertices take only two values.
    \end{enumerate}
    We say that $ij$ is valid for $y$ if it is not invalid.
\end{definition}

\begin{lemma}\label{lem:invalid-diagonal}
Let $y = (y_0, \dots, y_m) \in X(m)$, and let $ij$ be a diagonal of $P_{m+1}$. The following conditions are equivalent.
\begin{enumerate}
    \item $ij$ is valid for $y$.
    \item There exists a triangulation $T$ such that $ij \in T$, and $y$ belongs to the corresponding cluster torus.
\end{enumerate}
\end{lemma}
\begin{proof}
If $ij$ is valid for $y$ then $y_i \neq y_j$, and in each sub-polygon determined by $ij$ the labels of the vertices take at least three values. By \cite[Proposition 5.1]{deep-locus}, the deep locus of the cluster variety associated to the sub-polygon consists of labelings of the vertices of the sub-polygon whose labels only take two values. So in each sub-polygon we do not have a deep point, and thus we can find triangulations of each of these polygons so that the conditions of Lemma \ref{lem:cluster-torus-triangulation} are satisfied. The union of these triangulations and the diagonal $ij$ forms a triangulation that satisfies (2). The converse is proved similarly.
\end{proof}

\subsection{Cluster tori over \texorpdfstring{$\F_2$}{F2}} Note that $\#\F_q\PP^1 = q+1$. Following graph-theoretic terminology, we will refer to elements of $\F_q\PP^1$ as \emph{colors}. Recall that a \emph{proper} coloring of a graph $G$ on a set $C$ is a function $c: \mathrm{vertices}(G) \to C$ such that $c(v_1) \neq c(v_2)$ if $v_1v_2$ is an edge of $G$. Thus, interpreting a triangulation $T$ of the $(m+1)$-gon $P_{m+1}$ as a graph on the vertices of the polygon, we obtain that the elements of the corresponding cluster tori correspond precisely to the proper colorings of $T$ with $q+1$ colors, subject to the condition that $c(v_0) = \zero, c(v_{m}) = \infty$.

In the case $q = 2$ an easy induction shows that, subject to the condition $c(v_0) = \zero, c(v_{m}) = \infty$, a triangulation $T$ of $P_{m+1}$ admits a \emph{unique} proper coloring by elements of $\F_2\PP^1$. We thus obtain a map
\begin{equation}\label{eq:seeds-to-points-triangulations}
c: \mathsf{triangulations}(P_{m+1}) \to X_{\F_2}(m),
\end{equation}
that is a special case of the map \eqref{eq:seeds-to-points-refined} from the introduction. 

\begin{definition}\label{def:valid-diagonal-triangulation}
    Let $T$ be a triangulation of a $P_{m+1}$-gon, and let $ij$ be a diagonal of $P_{m+1}$. We say that $ij$ is valid for $T$ if it is valid for $c(T)$, cf. Definition \ref{def:valid-diagonal-point}. Else we say that $ij$ is invalid for $T$.
\end{definition}

Obviously, every diagonal of $T$ is valid for $T$, but there may be diagonals not in $T$ that are valid for $T$, see Figure \ref{fig:triangulations-hexagon}.

\subsection{The case of \texorpdfstring{$\A_3$}{A3}} For $m = 3,4$ (that is, the cases of cluster types $\A_1$ and $\A_2$) the map \eqref{eq:seeds-to-points-triangulations} is easily seen to be injective. Let us explicitly obtain the map \eqref{eq:seeds-to-points-triangulations} in the case of $\A_3$, that is, $X_{\F_2}(5)$. We record the $14$ triangulations of the hexagon with their corresponding images in $X_{\F_2}(5)$ in Figure \ref{fig:triangulations-hexagon}.

\begin{figure}[ht]
\begin{tikzpicture}
   \newdimen\R
\R=1cm
   \draw (0:\R)
   \foreach \x in {60,120,...,360} {  -- (\x:\R) };
   \draw[color=purple] (300:\R)--(240:\R);
   \draw (240:\R)--(360:\R);
   \draw (240:\R)--(60:\R);
   \draw (240:\R) -- (120:\R);
   \node at (240:1.2cm) {\tiny{$\zero$}};
   \node at (300:1.2cm) {\tiny{$\infty$}};
   \node at (180:1.2cm) {\tiny{$\infty$}};
   \node at (120:1.2cm) {\tiny{$\one$}};
   \node at (60:1.2cm) {\tiny{$\infty$}};
   \node at (360:1.2cm) {\tiny{$\one$}};

\begin{scope}[shift={(3,0)}]
   \draw (0:\R)
   \foreach \x in {60,120,...,360} {  -- (\x:\R) };
    \draw[color=purple] (300:\R)--(240:\R);
    \draw (180:\R) -- (300:\R);
    \draw (180:\R) -- (360:\R);
    \draw (180:\R) -- (60:\R);
     \node at (240:1.2cm) {\tiny{$\zero$}};
   \node at (300:1.2cm) {\tiny{$\infty$}};
   \node at (180:1.2cm) {\tiny{$\one$}};
   \node at (120:1.2cm) {\tiny{$\zero$}};
   \node at (60:1.2cm) {\tiny{$\infty$}};
   \node at (360:1.2cm) {\tiny{$\zero$}};
\end{scope}

\begin{scope}[shift={(6,0)}]
   \draw (0:\R)
   \foreach \x in {60,120,...,360} {  -- (\x:\R) };
    \draw[color=purple] (300:\R)--(240:\R);
    \draw (120:\R) -- (360:\R);
    \draw (120:\R)--(300:\R);
    \draw (120:\R) -- (240:\R);
     \node at (240:1.2cm) {\tiny{$\zero$}};
   \node at (300:1.2cm) {\tiny{$\infty$}};
     \node at (180:1.2cm) {\tiny{$\infty$}};
   \node at (120:1.2cm) {\tiny{$\one$}};
   \node at (60:1.2cm) {\tiny{$\infty$}};
   \node at (360:1.2cm) {\tiny{$\zero$}};
\end{scope}

\begin{scope}[shift={(9,0)}]
   \draw (0:\R)
   \foreach \x in {60,120,...,360} {  -- (\x:\R) };
    \draw[color=purple] (300:\R)--(240:\R);
    \draw (60:\R) -- (300:\R);
    \draw (60:\R) -- (240:\R);
    \draw(60:\R) -- (180:\R);
     \node at (240:1.2cm) {\tiny{$\zero$}};
   \node at (300:1.2cm) {\tiny{$\infty$}};
     \node at (180:1.2cm) {\tiny{$\infty$}};
   \node at (120:1.2cm) {\tiny{$\zero$}};
   \node at (60:1.2cm) {\tiny{$\one$}};
   \node at (360:1.2cm) {\tiny{$\zero$}};
\end{scope}

\begin{scope}[shift={(12,0)}]
   \draw (0:\R)
   \foreach \x in {60,120,...,360} {  -- (\x:\R) };
    \draw[color=purple] (300:\R)--(240:\R);
    \draw (360:\R) -- (240:\R);
    \draw (360:\R) -- (180:\R);
    \draw (360:\R) -- (120:\R);
     \node at (240:1.2cm) {\tiny{$\zero$}};
   \node at (300:1.2cm) {\tiny{$\infty$}};
     \node at (180:1.2cm) {\tiny{$\infty$}};
   \node at (120:1.2cm) {\tiny{$\zero$}};
   \node at (60:1.2cm) {\tiny{$\infty$}};
   \node at (360:1.2cm) {\tiny{$\one$}};
\end{scope}

\begin{scope}[shift={(0,-2.5)}]
   \draw (0:\R)
   \foreach \x in {60,120,...,360} {  -- (\x:\R) };
    \draw[color=purple] (300:\R)--(240:\R);
    \draw (300:\R) -- (60:\R);
    \draw (300:\R) -- (120:\R);
    \draw (300:\R) -- (180:\R);
     \node at (240:1.2cm) {\tiny{$\zero$}};
   \node at (300:1.2cm) {\tiny{$\infty$}};
     \node at (180:1.2cm) {\tiny{$\one$}};
   \node at (120:1.2cm) {\tiny{$\zero$}};
   \node at (60:1.2cm) {\tiny{$\one$}};
   \node at (360:1.2cm) {\tiny{$\zero$}};
\end{scope}

\begin{scope}[shift={(3,-2.5)}]
   \draw (0:\R)
   \foreach \x in {60,120,...,360} {  -- (\x:\R) };
    \draw[color=purple] (300:\R)--(240:\R);
    \draw (60:\R)--(300:\R)--(120:\R)--(240:\R);
     \node at (240:1.2cm) {\tiny{$\zero$}};
   \node at (300:1.2cm) {\tiny{$\infty$}};
     \node at (180:1.2cm) {\tiny{$\infty$}};
   \node at (120:1.2cm) {\tiny{$\one$}};
   \node at (60:1.2cm) {\tiny{$\zero$}};
   \node at (360:1.2cm) {\tiny{$\one$}};

\end{scope}

\begin{scope}[shift={(6,-2.5)}]
   \draw (0:\R)
   \foreach \x in {60,120,...,360} {  -- (\x:\R) };
    \draw[color=purple] (300:\R)--(240:\R);
      \draw (60:\R)--(180:\R)--(360:\R)--(240:\R);
     \node at (240:1.2cm) {\tiny{$\zero$}};
   \node at (300:1.2cm) {\tiny{$\infty$}};
        \node at (180:1.2cm) {\tiny{$\infty$}};
   \node at (120:1.2cm) {\tiny{$\one$}};
   \node at (60:1.2cm) {\tiny{$\zero$}};
   \node at (360:1.2cm) {\tiny{$\one$}};
\end{scope}

\draw[color=red] (1.5, -1.3) -- (7.5, -1.3) -- (7.5, -3.8) -- (1.5, -3.8) -- cycle;

\begin{scope}[shift={(9,-2.5)}]
   \draw (0:\R)
   \foreach \x in {60,120,...,360} {  -- (\x:\R) };
    \draw[color=purple] (300:\R)--(240:\R);
     \draw (360:\R)--(120:\R)--(300:\R)--(180:\R);
     \node at (240:1.2cm) {\tiny{$\zero$}};
   \node at (300:1.2cm) {\tiny{$\infty$}};
    \node at (180:1.2cm) {\tiny{$\one$}};
   \node at (120:1.2cm) {\tiny{$\zero$}};
   \node at (60:1.2cm) {\tiny{$\infty$}};
   \node at (360:1.2cm) {\tiny{$\one$}};
\end{scope}

\draw[color=red] (7.6, -1.3) -- (13.6, -1.3) -- (13.6, -3.8) -- (7.6, -3.8) -- cycle;

\begin{scope}[shift={(12,-2.5)}]
   \draw (0:\R)
   \foreach \x in {60,120,...,360} {  -- (\x:\R) };
    \draw[color=purple] (300:\R)--(240:\R);
    \draw (360:\R) -- (240:\R) -- (60:\R) -- (180:\R);
     \node at (240:1.2cm) {\tiny{$\zero$}};
   \node at (300:1.2cm) {\tiny{$\infty$}};
     \node at (180:1.2cm) {\tiny{$\one$}};
   \node at (120:1.2cm) {\tiny{$\zero$}};
   \node at (60:1.2cm) {\tiny{$\infty$}};
   \node at (360:1.2cm) {\tiny{$\one$}};
\end{scope}

\begin{scope}[shift={(1,-5)}]
   \draw (0:\R)
   \foreach \x in {60,120,...,360} {  -- (\x:\R) };
    \draw[color=purple] (300:\R)--(240:\R);
     \draw (120:\R)--(240:\R)--(60:\R)--(300:\R);
     \node at (240:1.2cm) {\tiny{$\zero$}};
   \node at (300:1.2cm) {\tiny{$\infty$}};
    \node at (180:1.2cm) {\tiny{$\one$}};
   \node at (120:1.2cm) {\tiny{$\infty$}};
   \node at (60:1.2cm) {\tiny{$\one$}};
   \node at (360:1.2cm) {\tiny{$\zero$}};
\end{scope}

\begin{scope}[shift={(4,-5)}]
   \draw (0:\R)
   \foreach \x in {60,120,...,360} {  -- (\x:\R) };
    \draw[color=purple] (300:\R)--(240:\R);
    \draw (120:\R) -- (360:\R) -- (180:\R) -- (300:\R);
     \node at (240:1.2cm) {\tiny{$\zero$}};
   \node at (300:1.2cm) {\tiny{$\infty$}};
       \node at (180:1.2cm) {\tiny{$\one$}};
   \node at (120:1.2cm) {\tiny{$\infty$}};
   \node at (60:1.2cm) {\tiny{$\one$}};
   \node at (360:1.2cm) {\tiny{$\zero$}};
\end{scope}

\draw[color=red] (-0.5, -3.8) -- (5.5, -3.8) -- (5.5, -6.3) -- (-0.5, -6.3) -- cycle;

\begin{scope}[shift={(7,-5)}]
   \draw (0:\R)
   \foreach \x in {60,120,...,360} {  -- (\x:\R) };
    \draw[color=purple] (300:\R)--(240:\R);
     \draw (60:\R)--(180:\R)--(300:\R)--(60:\R);
     \node at (240:1.2cm) {\tiny{$\zero$}};
   \node at (300:1.2cm) {\tiny{$\infty$}};
     \node at (180:1.2cm) {\tiny{$\one$}};
   \node at (120:1.2cm) {\tiny{$\infty$}};
   \node at (60:1.2cm) {\tiny{$\zero$}};
   \node at (360:1.2cm) {\tiny{$\one$}};
  
\end{scope}

\begin{scope}[shift={(10,-5)}]
   \draw (0:\R)
   \foreach \x in {60,120,...,360} {  -- (\x:\R) };
    \draw[color=purple] (300:\R)--(240:\R);
     \node at (240:1.2cm) {\tiny{$\zero$}};
   \node at (300:1.2cm) {\tiny{$\infty$}};
        \node at (180:1.2cm) {\tiny{$\one$}};
   \node at (120:1.2cm) {\tiny{$\infty$}};
   \node at (60:1.2cm) {\tiny{$\zero$}};
   \node at (360:1.2cm) {\tiny{$\one$}};
   \draw (120:\R)--(240:\R)--(360:\R)--(120:\R);
\end{scope}
\draw[color=red] (5.6, -3.8) -- (11.6, -3.8) -- (11.6, -6.3) -- (5.6, -6.3) -- cycle;
\end{tikzpicture}
\caption{All $14$ triangulations of the hexagon together with their images in $X_{\F_2}(5)$ under the map \eqref{eq:seeds-to-points-triangulations}. The red boxes indicate the triangulations that map to the same point in $X_{\F_2}(5)$.}
\label{fig:triangulations-hexagon}
\end{figure}
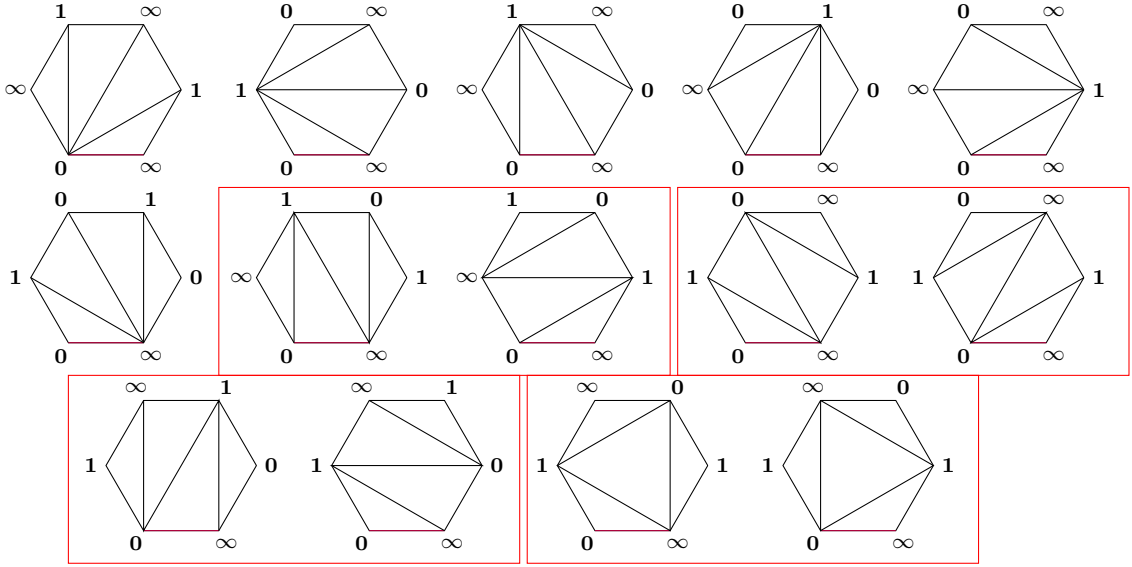

\subsection{Hexagonal moves}
Motivated by Figure \ref{fig:triangulations-hexagon}, we define the following moves on triangulations. 

\begin{definition}
Let $T$ be a triangulation of a $P_{m+1}$-gon. A hexagonal move on $T$ is a local move on $T$ as indicated by Figure \ref{fig:hexagonal-moves-text}.
\end{definition}

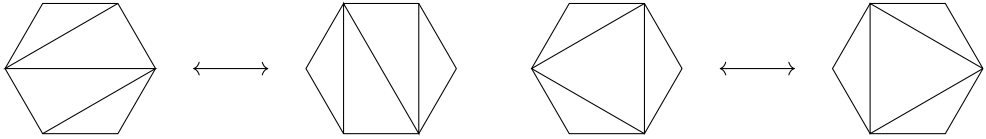
\begin{figure}
\begin{center}
\begin{tikzpicture}
   \newdimen\R
\R=1cm
   \draw (0:\R)
   \foreach \x in {60,120,...,360} {  -- (\x:\R) };
   \draw (60:\R)--(180:\R)--(360:\R)--(240:\R);
\draw[<->] (1.5,0) -- (2.5, 0);
\begin{scope}[shift={(4,0)}]
   \draw (0:\R)
   \foreach \x in {60,120,...,360} {  -- (\x:\R) };
   \draw (60:\R)--(300:\R)--(120:\R)--(240:\R);
\end{scope}
\begin{scope}[shift={(7,0)}]
   \draw (0:\R)
   \foreach \x in {60,120,...,360} {  -- (\x:\R) };
   \draw (60:\R)--(180:\R)--(300:\R)--(60:\R);
\end{scope}
\draw[<->] (8.5,0) -- (9.5,0);
\begin{scope}[shift={(11,0)}]
   \draw (0:\R)
   \foreach \x in {60,120,...,360} {  -- (\x:\R) };
   \draw (360:\R)--(120:\R)--(240:\R)--(360:\R);
\end{scope}
\end{tikzpicture}
\end{center}
\caption{Hexagonal moves. We refer to the move on the left-hand side as a \emph{zig-zag move}. It exchanges the zig-zag triangulation of a hexagon joining two antipodal points, with the other zig-zag joining the same antipodal points. We refer to the move on the right-hand side as a \emph{inscribed triangle} move.}
\label{fig:hexagonal-moves-text}
\end{figure}

For examples of hexagonal moves, see Figure \ref{fig:hexagonal-moves}. 

\begin{figure}
\begin{center}
\begin{tikzpicture}
   \newdimen\R
\R=1cm
   \draw (0:\R)
   \foreach \x in {40,80,120,...,360} {  -- (\x:\R) };
     \filldraw[color=lightgray] (360:\R)--(40:\R)--(80:\R)--(120:\R)--(160:\R)--(320:\R)--(360:\R);
     \draw (320:\R)--(360:\R)--(40:\R)--(80:\R)--(120:\R)--(160:\R);
   \draw (240:\R)--(160:\R)--(320:\R) -- (240:\R);
   \draw(120:\R)--(320:\R)--(80:\R)--(360:\R);

 \draw[<->] (1.5,0) -- (2.5, 0);
   \begin{scope}[shift={(4,0)}]
   \draw (0:\R)
  \foreach \x in {40,80,120,...,360} {  -- (\x:\R) };
    \filldraw[color=lightgray] (360:\R)--(40:\R)--(80:\R)--(120:\R)--(160:\R)--(320:\R)--(360:\R);
    \draw (320:\R)--(360:\R)--(40:\R)--(80:\R)--(120:\R)--(160:\R);
   \draw (240:\R)--(160:\R)--(320:\R) -- (240:\R);
   \draw(120:\R)--(40:\R)--(160:\R)--(360:\R);
\end{scope}

\begin{scope}[shift={(-4,0)}]
   \draw (0:\R)
   \foreach \x in {40,80,120,...,360} {  -- (\x:\R) };
   \filldraw[color=lightgray] (320:\R)--(280:\R)--(240:\R)--(200:\R)--(160:\R) -- (120:\R) --(320:\R);
   \draw (320:\R)--(280:\R)--(240:\R)--(200:\R)--(160:\R)--(120:\R);
   \draw (240:\R)--(160:\R)--(320:\R) -- (240:\R);
   \draw(120:\R)--(320:\R)--(80:\R)--(360:\R);
\end{scope}

\draw[<->] (-6.5,0) -- (-5.5, 0);

\begin{scope}[shift={(-8,0)}]
   \draw (0:\R)
   \foreach \x in {40,80,120,...,360} {  -- (\x:\R) };
   \filldraw[color=lightgray] (320:\R)--(280:\R)--(240:\R)--(200:\R)--(160:\R) -- (120:\R) --(320:\R);
   \draw (320:\R)--(280:\R)--(240:\R)--(200:\R)--(160:\R)--(120:\R);
\draw(200:\R)--(120:\R)--(280:\R)--(200:\R);
   \draw(120:\R)--(320:\R)--(80:\R)--(360:\R);
\end{scope}
\end{tikzpicture}
\end{center}
\caption{Hexagonal moves on triangulations of a $9$-gon. The hexagon where the move is happening is shaded.}
\label{fig:hexagonal-moves}
\end{figure}
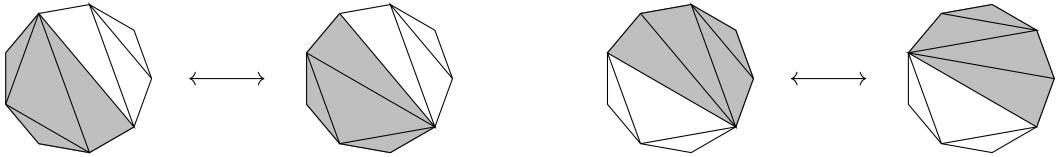

\begin{theorem}\label{thm:main-text}
    Let $T$ and $T'$ be triangulations of the polygon $P_{m+1}$. Then, $c(T) = c(T')$ if and only if $T$ and $T'$ are related through a sequence of hexagonal moves, where $c$ is the map from \eqref{eq:seeds-to-points-triangulations}.
\end{theorem}

\begin{proof}
If $T$ and $T'$ are related by a hexagonal move then our computation in Figure \ref{fig:triangulations-hexagon} implies that $c(T) = c(T')$.

For the converse direction, we proceed by induction on $m$. The case $m = 5$ is Figure \ref{fig:triangulations-hexagon}, so let $m > 5$ and assume the result to be true for $5\leq k < m$. Let $T$ and $T'$ be two triangulations mapping to the same point in $X_{\F_2}(m)$.  We separate in several cases.

\vspace{1em}

   {\bf Case 1. $T$ and $T'$ share a diagonal.} In this case, the shared diagonal divides $P_{m+1}$ into two smaller polygons. These polygons inherit triangulations $T_1, T_2$ and $T'_1,T'_2$ from $T$ and $T'$ respectively. We have that $T_1$ and $T_1'$ share the same coloring, and so do $T_2$ and $T_2'$. Then by induction hypothesis there exist sequences of hexagonal moves that taking $T_1$ to $T_1'$ and $T_2$ to $T_2'$. These two sequences combine to take $T$ to $T'$ which proves the result in this case. 

\vspace{1em}

{\bf Case 2. $T$ and $T'$ do not share diagonals.}  Let $ij$ be a diagonal of $T$ such that one of the two polygons in which this diagonal divides $P_{m+1}$ has at least $6$ vertices, we call this polygon $P_1$. Let $T_1$ be the triangulation of this polygon inherited from $T$. We consider two further sub-cases. 

\vspace{1em}

{\bf Case 2.1. There is a diagonal $kl$ in $T'$  that does not cross the diagonal $ij$.} Note that if $kl$ is not contained in $P_1$, and the complementary polygon to $P_1$ has less than 5 vertices then, since the map \eqref{eq:seeds-to-points-triangulations} is injective for $m < 5$, the diagonal $kl$ is shared by $T$ and $T'$, a contradiction. So we may assume without loss of generality that $kl$ is contained in $P_1$.  Now we consider two sub-cases, depending on whether $kl$ is valid for $T_1$, cf. Definition \ref{def:valid-diagonal-triangulation}. 

\vspace{1em}

{\bf Case 2.1.1. The diagonal $kl$ is a valid diagonal for $T_1$.} Then there exists another triangulation $T_2$ of $P_1$ that contains $kl$ and so that $c(T_1) = c(T_2)$. So by induction, there is a sequence of hexagonal moves that takes $T_1$ into $T_2$, and $kl$ is a diagonal in $T_2$. After this we get a new triangulation $T''$ formed by $T_2$ and the triangulation inherited from $T$ on the complement to the polygon $P_1$. This triangulation $T''$ can be obtained from $T$ by a sequence of hexagonal moves and $T'$, $T''$ share the diagonal $kl$. So  by Case 1 above we can take $T''$ to $T'$ by a sequence of hexagonal moves and the result follows.

\vspace{1em}

{\bf Case 2.1.2. The diagonal $kl$ is not a valid diagonal for $T_1$.} We consider the polygon delimited by $i, j, k, l$. We claim that only two different colors appear in the $y$-values of this polygon. This follows by observing that, while $kl$ is not valid for $T_1$ it is valid for $T$ because $c(T) = c(T')$ and $kl$ is a diagonal of $T'$. So in order for $kl$ to satisfy Definition \ref{def:valid-diagonal-point} we must have that it satisfies (2) on the polygon truncated by $ij$.

        We can assume that both $ij$ and $kl$ leave only one vertex on one of the triangulations in which they divide the polygon. We verify this on Figure \ref{fig:hex-move1}.

        \begin{figure}[ht]
            \begin{center}
                \includegraphics[scale = 1.2]{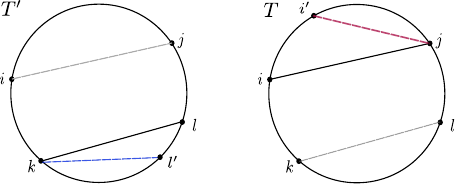}
            \end{center}
            \caption{The diagonal $ij$ belongs to $T$, and the diagonal $kl$ belongs to $T'$. If the diagonal $kl$ does not leave only one vertex at one of its sides, we can pick a diagonal adjacent to $k$ or $l$ as in the left-hand side of the figure. We can then take the diagonal $kl'$ instead of the diagonal $kl$. If there are more than three colors in the polygon delimited by $ijl'k$, then $ij$ is a valid diagonal for the triangulation on the polygon delimited by $kl'$ and including the diagonal $ij$ inherited from $T'$. So by induction we can find a sequence of hexagonal moves on this polygon that take $T'$ to a triangulation containing $ij$. A similar consideration applies if the diagonal $ij$ does not leave only one vertex on one of its sides, as in the right-hand side of this figure.}
            \label{fig:hex-move1}
        \end{figure}

        After this reduction the polygon delimited by $ij$ and $kl$ does not have a proper coloring, so the vertices between these two edges are colored just by two colors in an alternating fashion. More so, these diagonals $ij$ and $kl$  each leave only one vertex at a side. So the triangulations $T$ and $T'$ are formed by a triangle and a fan, as shown in Figure \ref{fig:hex-move-2}.
        
        \begin{figure}[ht]
            \begin{center}
                \includegraphics[scale = 1.2]{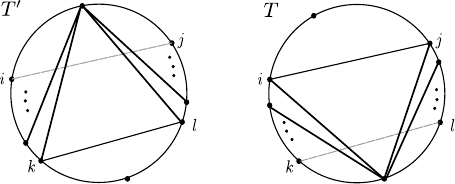}
            \end{center}
            \caption{In Case 2.1.2, after reductions we may assume the triangulations $T$ and $T'$ have shapes as indicated in the figure.}
            \label{fig:hex-move-2}
        \end{figure}

        Note that the vertices left outside of the polygon $ijlk$ must be of the same color, which is precisely the color missing in the polygon $ijlk$. Now, we do an hexagonal flip on $T$ like the one shown in Figure \ref{fig:hex-move-3}.

        \begin{figure}[ht]
        \begin{center}
            \includegraphics[scale = 1.2]{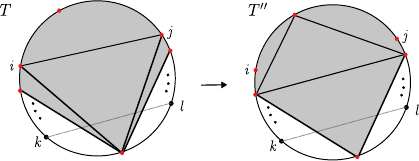}
        \end{center}
        \caption{From Figure \ref{fig:hex-move-2}, after doing an hexagonal flip on the hexagon determined by the red vertices, we have a triangulation sharing a diagonal with $T'$.}
        \label{fig:hex-move-3}
        \end{figure}

        After this flip we get a triangulation $T''$ that shares at least one diagonal with $T'$. From Case 1 we already know that there is a sequence of flips that takes $T''$ to $T'$, and this finishes this case.
        
        \vspace{1em}

        {\bf Case 2.2. Every diagonal in $T'$ crosses the diagonal $ij$.} Fix a diagonal $ab$ of $T$ that does not cross $ij$. By how we took $T'$ the vertices $a$ and $b$ are only adjacent to diagonals of $T'$ that cross the diagonal $ij$. Take diagonals $ar$ and $bs$ of $T'$ adjacent to $a$ and $b$. Note that the polygon $P''$ delimited by the diagonal $ar$ that includes the vertices $a,b,j,s$ excludes $i$ and so it has strictly less vertices than $P$. This is shown in Figure \ref{fig:hex-move-4}.

        \begin{figure}
        \begin{center}
            \includegraphics[scale = 1.2]{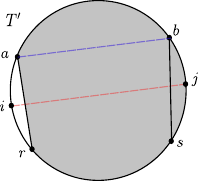}
        \end{center}
        \caption{The colored diagonals $ab$ and $ij$ belong to $T$, while the diagonals $ar$ and $bs$ belong to $T'$. The diagonal $ab$ is valid on the shaded $P''$ since $ab$ is valid for $T$ and $bs$ for $T''$.}
        \label{fig:hex-move-4}
        \end{figure}

       Let $T''$ be the triangulation on $P''$ inherited from $T'$. We verify that $ab$ is valid on $T''$. We follow Figure \ref{fig:hex-move-4}. Since $ab$ is valid for $T$, there are at least three colors on the top of $ab$, so we only need to verify that there are three colors on the bottom of $ab$, excluding the vertices not belonging to $P''$. This follows by noting that $bs$ is a valid diagonal $T'$, so the vertices on the right-hand side of this diagonal are colored by $3$ colors, and these all stay below $ab$. So the vertices above $ab$ have $3$ colors, and so do the vertices below $ab$. This verifies what we wanted.

        As $ab$ is valid on $T''$, by induction there exists a sequence of hexagonal flips that takes $T''$ to a new triangulation that has the diagonal $ab$. As $T''$ is contained on $T'$ this sequence takes $T'$ to a triangulation $T'''$ that shares $ab$ with $T$. From this the result follows.
\end{proof}

\begin{example}
    Consider the following triangulations on a $12$-gon.
    \begin{center}
        \begin{tikzpicture}
               \newdimen\R
\R=1cm
   \draw (0:\R)
   \foreach \x in {30,60,90,120,...,360} {  -- (\x:\R) };
   \draw[color=purple] (270:\R)--(300:\R);
   \node at (270:1.1cm) {\tiny{$\zero$}};
   \node at (300:1.1cm) {\tiny{$\infty$}};
   \node at (330:1.1cm) {\tiny{$\one$}};
   \node at (360:1.1cm) {\tiny{$\zero$}};
   \node at (30:1.1cm) {\tiny{$\infty$}};
   \node at (60:1.1cm) {\tiny{$\one$}};
   \node at (90:1.1cm) {\tiny{$\zero$}};
   \node at (120:1.1cm) {\tiny{$\infty$}};
   \node at (150:1.1cm) {\tiny{$\one$}};
   \node at (180:1.1cm) {\tiny{$\zero$}};
   \node at (210:1.1cm) {\tiny{$\infty$}};
   \node at (240:1.1cm) {\tiny{$\one$}};

\draw(270:\R)--(30:\R)--(330:\R)--(270:\R);
\draw(240:\R)--(30:\R)--(150:\R); \draw(180:\R)--(30:\R); \draw(180:\R)--(240:\R);
\draw(90:\R)--(150:\R); \draw(90:\R)--(30:\R);

\begin{scope}[shift={(4,0)}]

   \draw (0:\R)
   \foreach \x in {30,60,90,120,...,360} {  -- (\x:\R) };
   \draw[color=purple] (270:\R)--(300:\R);
   \node at (270:1.1cm) {\tiny{$\zero$}};
   \node at (300:1.1cm) {\tiny{$\infty$}};
   \node at (330:1.1cm) {\tiny{$\one$}};
   \node at (360:1.1cm) {\tiny{$\zero$}};
   \node at (30:1.1cm) {\tiny{$\infty$}};
   \node at (60:1.1cm) {\tiny{$\one$}};
   \node at (90:1.1cm) {\tiny{$\zero$}};
   \node at (120:1.1cm) {\tiny{$\infty$}};
   \node at (150:1.1cm) {\tiny{$\one$}};
   \node at (180:1.1cm) {\tiny{$\zero$}};
   \node at (210:1.1cm) {\tiny{$\infty$}};
   \node at (240:1.1cm) {\tiny{$\one$}};

\draw (270:\R)--(120:\R)--(240:\R)--(180:\R)--(120:\R);
\draw (270:\R)--(60:\R)--(120:\R);
\draw(60:\R)--(300:\R)--(0:\R)--(60:\R);

\end{scope}
        \end{tikzpicture}
    \end{center}

We have the following sequence of hexagonal moves converting one to the other. In each triangulation, we shade the hexagon where the move is to be performed.

    \begin{center}
        \begin{tikzpicture}
               \newdimen\R
\R=1cm
   \draw (0:\R)
   \foreach \x in {30,60,90,120,...,360} {  -- (\x:\R) };
   \filldraw[color=lightgray] (240:\R)--(30:\R)--(360:\R)--(330:\R)--(300:\R)--(270:\R)--(240:\R);
   \draw  (240:\R)--(30:\R)--(360:\R)--(330:\R)--(300:\R)--(270:\R)--(240:\R);
\draw(270:\R)--(30:\R)--(330:\R)--(270:\R);
\draw(240:\R)--(30:\R)--(150:\R); \draw(180:\R)--(30:\R); \draw(180:\R)--(240:\R);
\draw(90:\R)--(150:\R); \draw(90:\R)--(30:\R);
 \draw[->] (1.2, 0) -- (1.8, 0);
\begin{scope}[shift={(3,0)}]
\draw (0:\R)
   \foreach \x in {30,60,90,120,...,360} {  -- (\x:\R) };
   \filldraw[color=lightgray] (180:\R)--(30:\R)--(60:\R)--(90:\R)--(120:\R)--(150:\R)--(180:\R);
   \draw (180:\R)--(30:\R)--(60:\R)--(90:\R)--(120:\R)--(150:\R)--(180:\R);
   \draw(240:\R)--(300:\R)--(360:\R)--(240:\R);
\draw(240:\R)--(30:\R)--(150:\R); \draw(180:\R)--(30:\R); \draw(180:\R)--(240:\R);
\draw(90:\R)--(150:\R); \draw(90:\R)--(30:\R);
\end{scope}

 \draw[->] (4.2, 0) -- (4.8, 0);

\begin{scope}[shift={(6,0)}]
\draw (0:\R)
   \foreach \x in {30,60,90,120,...,360} {  -- (\x:\R) };
   \filldraw[color=lightgray] (240:\R)--(180:\R)--(120:\R)--(60:\R)--(30:\R)--(360:\R)--(240:\R);
   \draw(60:\R)--(30:\R)--(360:\R);
   \draw(240:\R)--(300:\R)--(360:\R)--(240:\R);
\draw(240:\R)--(30:\R); \draw(180:\R)--(30:\R); 
\draw(180:\R)--(240:\R);
\draw(120:\R)--(180:\R)--(60:\R)--(120:\R);
\end{scope}

 \draw[->] (7.2, 0) -- (7.8, 0);

\begin{scope}[shift={(9,0)}]
\draw (0:\R)
   \foreach \x in {30,60,90,120,...,360} {  -- (\x:\R) };
   \filldraw[color=lightgray] (360:\R)--(60:\R)--(120:\R)--(240:\R)--(270:\R)--(300:\R)--(360:\R);
   \draw(240:\R)--(300:\R)--(360:\R)--(240:\R);
   \draw(240:\R)--(270:\R)--(300:\R);
\draw(180:\R)--(240:\R);
\draw(120:\R)--(180:\R); \draw(60:\R)--(120:\R);
\draw (240:\R)--(120:\R)--(360:\R)--(60:\R);
\end{scope}

 \draw[->] (10.2, 0) -- (10.8, 0);

\begin{scope}[shift={(12,0)}]
\draw (0:\R)
   \foreach \x in {30,60,90,120,...,360} {  -- (\x:\R) };
\draw (270:\R)--(120:\R)--(240:\R)--(180:\R)--(120:\R);
\draw (270:\R)--(60:\R)--(120:\R);
\draw(60:\R)--(300:\R)--(0:\R)--(60:\R);
\end{scope}
\end{tikzpicture}
\end{center}
\end{example}

\begin{remark}
    The hexagonal moves can be interpreted as edges in the universal polytope of triangulations. Note, however, that edges of this polytope also include the usual (quadrilateral) flips and many other moves. See e.g. \cite[8.5.1]{triangulations-book}. 
\end{remark}

\begin{corollary}
Let $y = (y_0, \dots, y_m) \in X_{\F_2}(m)$ be such that $y_i = \one$ for at least one $i$. Then, $\#c^{-1}(y) = 1$ if and only if there exists $i = 0, \dots, m$ such that $y_i \neq y_j$ for every $j \neq i$.
\end{corollary}
\begin{proof}
    Note that no hexagonal moves can be applied to a triangulation $T$ if and only if $T$ is a fan triangulation.  If $i$ is the vertex that is incident to every diagonal in a fan triangulation, then $y_i \neq y_j$ for every $j \neq i$. 
\end{proof}

We do not have a formula for $\#c^{-1}(y)$ for arbitrary $y \in X_{\F_2}(m)$.

\section{Minimal coverings of type \texorpdfstring{$\A$}{A} cluster varieties}\label{sec:minimal-coverings}

\subsection{Definitions} Recall that a subset $C \subseteq \seeds(A)$ is said to be an $\F$-covering if
\[
\bigcup_{\seed = (\x, Q) \in C} \mathbb{T}_{\x} = \manifold_{\F},
\]
where $\manifold_{\F}$ is the cluster manifold defined over $\F$. We say that the $\F$-covering $C$ is minimal if, moreover, for every $D \subsetneq C$, we have
\[
\bigcup_{\seed = (\x, Q) \in D} \T_{\x} \subsetneq \manifold_{\F}.  
\]

Clearly, a subset $C \subseteq \seeds(A)$ is $\F_2$-minimal if and only if it contains the image of a section of the map \eqref{eq:seeds-to-points-refined}, and the $\F_2$-minimal coverings are precisely such images of sections. Our goal for this section is, in type $\A$, to provide an example of an $\F_2$-minimal covering that is a universal minimal covering, that is, it is an $\F$-minimal covering for any field $\F$. On the other hand, we will show the existence of an $\F_2$-covering that is not an $\F$-covering for any other field $\F$.

\subsection{A minimal covering} For brevity, we will denote
\[
X'_{\F}(m) := X_{\F}(m) \setminus \{(\zero, \infty, \zero, \infty, \dots \zero, \infty)\},
\]
that is the union of all cluster tori in $X_{\F}(m)$. Note that $X_{\F}(m) = X'_{\F}(m)$ if $m$ is even.

For any field $\F$ we denote
\[
\one := [1:1] \in \F\PP^1
\]
and use this to embed $X_{\F_2}(m)$ into $X_{\F}(m)$ for any $m$. Note that under this embedding we have $X'_{\F_2}(m) \subseteq X'_{\F}(m)$. 

\begin{proposition}\label{prop:covering-1}
    For any $m>0$ and any field $\F$ there exists a minimal covering of $X'_{\F}(m)$ with as many elements as $X'_{\F_2}(m)$.
\end{proposition}

\begin{proof}
Throughout this proof, we denote $X'(m) := X_{\F}(m)$.    Let $y = (y_0, \dots, y_m) \in X'(m)$. As before, we consider a polygon $P_{m+1}$ with $m+1$ vertices, labeled clockwise by the elements $y_0, \dots, y_{m}$. We will construct a triangulation $T$ such that $y$ is in the corresponding cluster torus. By definition of $X'(m)$ there exists $i>0$ such that $y_i \neq \zero,\infty$. We apply the following algorithm. 

\vspace{1em}

{\underline{\textbf{Algorithm A.}}}

\noindent \underline{Input:} $y = (y_0, \dots, y_m) \in X'(m)$.

\noindent \underline{Output:} A triangulation $T$ of $P_{m+1}$, such that $y$ is a valid proper coloring for $T$. 

\begin{itemize}
    \item[Step 1.] Find the minimal $i_1$ such that $y_{i_1} \neq \zero, \infty$. Draw diagonals between the vertex $i_1$ and $m, 0, 1, \dots, i_1-2, i_1-1$. Note that by construction $y_{i_1} \neq y_j$ for $j = m, 0, 1, \dots, i-1$, so these are all valid diagonals for $y$.
    \item[Step 2.] Find the minimal $i_2 > i_1$ such that $y_{i_2} = \zero$. Draw diagonals from $i_2$ to $m, i_1, i_1 + 1, \dots, i_2 -1$. Again, these are all valid diagonals for $y$. 
    \item[Step 3.] Find the minimal $i_3 > i_2$ such that $y_{i_3} \neq \zero, \infty$, and draw diagonals from $i_3$ to $m$, $i_2, i_2 + 1, \dots, i_3 - 1$.
\end{itemize}
We continue with this procedure (finding $i_k > i_{k-1}$ such that $y_{i_k} \neq \zero, \infty$ if $k$ is odd, and such that $y_{i_k} = \zero$ if $k$ is even) until we either construct a triangulation of $P_{m+1}$ or it is no longer possible to continue. There are two reasons why it may not be possible to continue. 

    \begin{itemize}
        \item[(1)] We have found $i_k$ such that $y_{i_k} \neq \zero, \infty$, but there is no $i_{k+1} > i_{k}$ such that $y_{i_{k+1}} = \zero$. In this case, we backtrack and delete all diagonals from $i_k$ to $m$ and from $i_k$ to $i_{k-1}, i_{k-1}+1, \dots$ we have drawn. We consider the largest $\ell$ such that $y_{\ell} = \zero$. Note that $i_{k-1} \leq \ell < i_{k}$. We instead draw diagonals from $\ell$ to $\ell +1, \ell + 2, \dots, m-1$. Note that these are all valid diagonals, since $y_j \neq \zero$ for $j > \ell$. We still have to triangulate the polygon with vertices $i_{k-1}, i_{k+1}, \dots, \ell, m-1$. But $y_{m-1} \neq \zero$, by assumption, and $y_{m-1} \neq y_{m} = \infty$, while all $y_{i_{k-1}}, \dots, y_{\ell} \in \{\zero, \infty\}$. So we draw diagonals from $m-1$ to $i_{k-1}, i_{k-1}+1, \dots, \ell$. After doing this, we end up with a triangulation of $P_{m+1}$. 

         \item[(2)] We have found $i_k$ such that $y_{i_k} = \zero$, but there is no $i_{k+1} > i_{k}$ such that $y_{i_{k+1}} \neq \zero, \infty$. We backtrack again and delete all the diagonals incident to $i_k$ that we have drawn. Instead, we draw diagonals from $i_{k-1}$ to $m, m-1, \dots, i_{k}$, these are all valid since $y_{i_{k-1}} \neq \zero, \infty$, and also draw diagonals from $i_{k}$ to $i_{k-1}, i_{k-1}+1, \dots, i_{k}-1$. These are all valid again since $y_{i_{k-1}}, \dots, y_{i_{k}-1} \neq \zero$. This gives triangulation of $P_{m+1}$ with $y$ as a proper coloring.
    \end{itemize}

    See Figure \ref{fig:examples-algorithm} for examples of this procedure. In the end, we obtain a triangulation that admits $y$ as a proper coloring, that is, $y$ belongs to the cluster torus defined by this triangulation. In other words, we have a function
    \[
    \function: X'(m) \to \mathsf{triangulations}(P_{m+1})
    \]
    whose image is a covering set. It remains to see that the image of $\function$ has exactly as many elements as $X'_{\F_2}(m)$. For this, we take the composition with the map $c$ from \eqref{eq:seeds-to-points-triangulations},
    \[
    c \circ \function: X'(m) \to X'_{\F_2}(m).
    \]
    We claim that 
    \begin{equation}\label{eq:composition}
    \function = \function \circ c \, \circ \function.
    \end{equation}Note that this implies the desired result, as then $\#\mathrm{image}(\function) \leq \#\mathrm{image}(c \circ\,\function) = \#X'_{\F_2}(m)$.

    To prove \eqref{eq:composition}, we first analyze the map $c \, \circ \function$ more carefully. This is a map $c\, \circ \function: X'(m) \to X'_{\F_2}(m)$. For this, we translate Algorithm A into the following. 
    
    \vspace{1em}

    {\underline{\textbf{Algorithm B.}}}

\noindent \underline{Input:} $y = (y_0, \dots, y_m) \in X'(m)$.

\noindent \underline{Output:} An element $z = (z_0, \dots, z_m) \in X'_{\F_2}(m)$. 

\begin{itemize}
    \item[Step 1.] Find the minimal $i_1$ such that $y_{i_1} \neq \zero, \infty$. Set $z_{i_1} = \one$, and $z_1 = y_1, z_2 = y_2, \dots, z_{i_1-1} = y_{i_1-1}$. 
    \item[Step 2.] Find the minimal $i_2 > i_1$ such that $y_{i_2} = \zero$. Set $z_{i_2} = y_{i_2} = \zero$, and $z_{i_1+1}, \dots, z_{i_2 - 1}$ alternating between $\one$ and $\infty$, starting with $z_{i_1+1} = \infty$. 
    \item[Step 3.] Find the minimal $i_3 > i_2$ such that $y_{i_3} \neq \zero, \infty$. Set $z_{i_3} = \one$, and $z_{i_2+1} = y_{i_2+1}, \dots, z_{i_3-1} = y_{i_3-1}$.
\end{itemize}
We continue with this procedure (finding $i_k > i_{k-1}$ such that $y_{i_k} \neq \zero, \infty$ if $k$ is odd, and such that $y_{i_k} = \zero$ if $k$ is even) until we either have set values of $z_i$ for all $i = 0, \dots, m$, or it is no longer possible to continue. There are two reasons why it may not be possible to continue. 

    \begin{itemize}
        \item[(1)] We have found $i_k$ such that $y_{i_k} \neq \zero, \infty$, but there is no $i_{k+1} > i_{k}$ such that $y_{i_{k+1}} = \zero$.  In this case, we backtrack and re-do the values of $z_{i_{k-1}+1}, \dots, z_{i_k}$ that we have found before. Instead, we consider the largest $\ell$ such that $y_{\ell} = \zero$. Note that $i_{k-1} \leq \ell < i_{k}$. Set $z_{i_{k-1}+1} = y_{i_{k-1}+1}, \dots, z_{\ell} = y_{\ell}$ (note that these are all $\zero$ or $\infty$), and set $z_{\ell+1}, \dots, z_{m}$ alternating between $\one$ and $\infty$ in such a way that $z_m = \infty$.

        \item[(2)] We have found $i_k$ such that $y_{i_k} = \zero$, but there is no $i_{k+1} > i_{k}$ such that $y_{i_{k+1}} \neq \zero, \infty$. As in (1) above, we backtrack and set $z_{i_{k-1}}, \dots, z_{i_{k}-1}$ alternating between $\one$ and $\infty$, starting with $z_{i_{k-1}} = \one$, while $z_{i_{k}} = \zero$. We also get $z_{i_{k}+1} = y_{i_{k}+1}, \dots, z_{m-1} = y_{m-1}, z_{m} = y_{m}$.
    \end{itemize}

    Comparing Algorithms A and B, it is clear that the output of Algorithm B is precisely $c \circ \function (y)$, see Figure \ref{fig:examples-algorithm-B}. 

    Now let $z = c \circ \function(y)$. We have to show that $\function(z) = \function(y)$. If do not have to perform one of the exceptions (1) or (2) in Algorithm B, we have that the numbers $i_1, i_2, \dots$ found in the first part of Algorithm A are the same for $y$ and for $z$, and this implies that $\function(y) = \function(z)$. 

    If we have to apply (1) for $y$, note that, since $z_{\ell} = \zero$ but $z_{\ell+1}, \dots, z_{m} \in \{\one, \infty\}$, then we also have to apply (1) for $z$, so $\function(y) = \function(z)$. Similarly, if we have to apply (2) for $y$ then we also have to apply (2) for $z$, and we also get $\function(y) = \function(z)$. Thus, $\#\mathrm{image}(\function) \leq \#|X'_{\F_2}(m)|$.
    Now, recall that we can embed $X'_{\F_2}(m)$ into $X'_{\F}(m)$ and, since a triangulation determines a unique point of $X'_{\F_2}(m)$, the restriction of $\function$ to $X'_{\F_2}(m)$ is injective. This implies that $\#\mathrm{image}(\function) \geq \#|X'_{\F_2}(m)|$ and finishes the proof.
    \end{proof}

    From the proof of Proposition \ref{prop:covering-1}, we obtain the following.

\begin{theorem}
   Let $m > 0$ and consider a cluster algebra of type $\A_m$ with really full rank. Then, there exists a set $C \subseteq \seeds(\A_m)$ that is a minimal $\F$-covering for any field $\F$.
\end{theorem}
\begin{proof}
Note that the collection $C := \function(X_{\F}(m))$ does not depend on the field $\F$. We claim it is a minimal covering of $X_{\F}(m)$ for any field $\F$. It is a covering by Proposition \ref{prop:covering-1}, and it is minimal since an element $z \in X_{\F_2}(m) \subseteq X_{\F}(m)$ can only be covered by $\function(z)$. 
The variety $X_{\F}(m)$ is a cluster variety of type $\A_{m-2}$ with really full rank and a single frozen. To pass to any cluster variety of type $\A_{m-2}$ and really full rank, we use Proposition 5.11 in \cite{LSI}: if $\var$ denotes any cluster variety of type $\A_{m-2}$ and really full rank, then there exist $n, n'$ such that $\var_\F \times (\F^{\times})^n \cong X_{\F}(m)\times (\F^{\times})^{n'}$; a covering of $X_{\F}(m)$ gives a covering of  $X_{\F}(m)\times (\F^{\times})^{n'} \cong \var \times (\F^{\times})^n$ that, after setting some monomials in frozen variables equal to $1$, gives a covering of $\var_{\F}$. 
\end{proof}

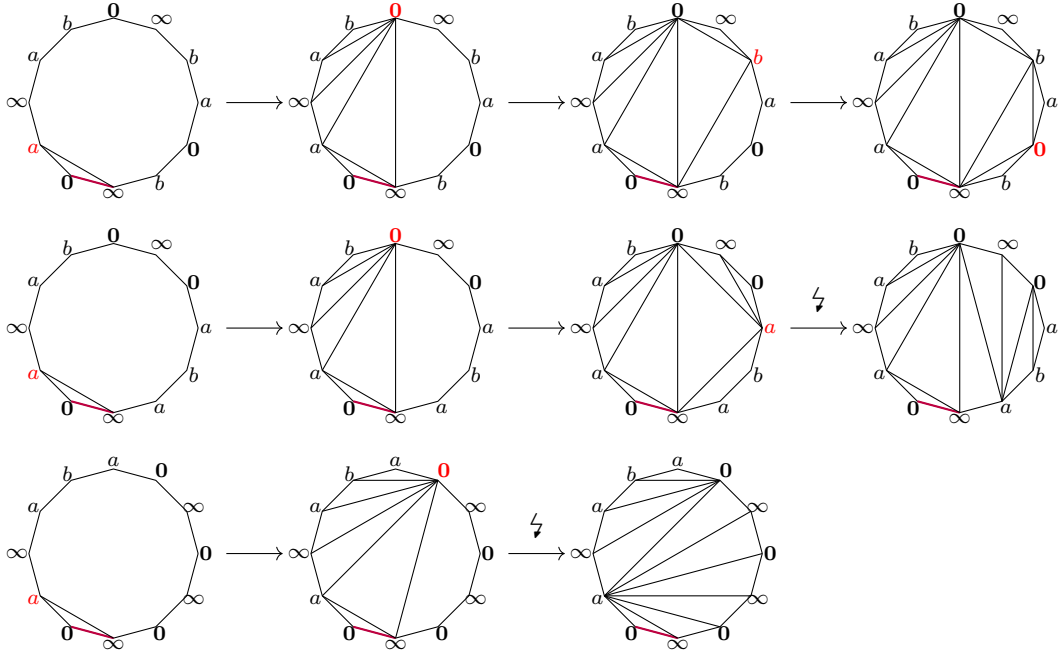
\begin{figure}
    \begin{center}
\begin{tikzpicture}[scale=0.7]
   \newdimen\R
\R=1.5cm
   \draw (0:\R)
   \foreach \x in {30,60,90,...,360} {  -- (\x:\R) };
   \draw[thick, color=purple] (240: \R) -- (270:\R);
   \node at (270:\R+3) {\tiny{$\infty$}};
   \node at (240:\R+4) {\tiny{$\zero$}};
   \node at (210:\R+4) {\color{red}\tiny{$a$}};
   \node at (180:\R+6) {\tiny{$\infty$}};
   \node at (150:\R+4) {\tiny{$a$}};
   \node at (120:\R+4) {\tiny{$b$}};
   \node at (90:\R+4) {\tiny{$\zero$}};
   \node at (60:\R+6) {\tiny{$\infty$}};
   \node at (30:\R+4) {\tiny{$b$}};
   \node at (360:\R+4) {\tiny{$a$}};
   \node at (330:\R+4) {\tiny{$\zero$}};
   \node at (300:\R+4) {\tiny{$b$}};
    \draw(210:\R)--(270:\R);

\draw[->] (2,0) -- (3,0);

\begin{scope}[shift={(5,0)}]

   \draw (0:\R)
   \foreach \x in {30,60,90,...,360} {  -- (\x:\R) };
   \draw[thick, color=purple] (240: \R) -- (270:\R);
   \node at (270:\R+3) {\tiny{$\infty$}};
   \node at (240:\R+4) {\tiny{$\zero$}};
   \node at (210:\R+4) {\tiny{$a$}};
   \node at (180:\R+6) {\tiny{$\infty$}};
   \node at (150:\R+4) {\tiny{$a$}};
   \node at (120:\R+4) {\tiny{$b$}};
   \node at (90:\R+4) {\color{red}\tiny{$\zero$}};
   \node at (60:\R+6) {\tiny{$\infty$}};
   \node at (30:\R+4) {\tiny{$b$}};
   \node at (360:\R+4) {\tiny{$a$}};
   \node at (330:\R+4) {\tiny{$\zero$}};
   \node at (300:\R+4) {\tiny{$b$}};
    \draw(210:\R)--(270:\R);
    \draw(90:\R) -- (270:\R);
    \draw(90:\R)--(210:\R);
    \draw(90:\R)--(180:\R);
    \draw(90:\R)--(150:\R);
    \draw[->] (2,0) -- (3,0);

    \begin{scope}[shift={(5,0)}]

   \draw (0:\R)
   \foreach \x in {30,60,90,...,360} {  -- (\x:\R) };
   \draw[thick, color=purple] (240: \R) -- (270:\R);
   \node at (270:\R+3) {\tiny{$\infty$}};
   \node at (240:\R+4) {\tiny{$\zero$}};
   \node at (210:\R+4) {\tiny{$a$}};
   \node at (180:\R+6) {\tiny{$\infty$}};
   \node at (150:\R+4) {\tiny{$a$}};
   \node at (120:\R+4) {\tiny{$b$}};
   \node at (90:\R+4) {\tiny{$\zero$}};
   \node at (60:\R+6) {\tiny{$\infty$}};
   \node at (30:\R+4) {\color{red}\tiny{$b$}};
   \node at (360:\R+4) {\tiny{$a$}};
   \node at (330:\R+4) {\tiny{$\zero$}};
   \node at (300:\R+4) {\tiny{$b$}};
    \draw(210:\R)--(270:\R);
    \draw(90:\R) -- (270:\R);
    \draw(90:\R)--(210:\R);
    \draw(90:\R)--(180:\R);
    \draw(90:\R)--(150:\R);
    \draw(30:\R)--(90:\R);
    \draw(30:\R)--(270:\R);
\draw[->] (2,0) -- (3,0);
    \begin{scope}[shift={(5,0)}]

   \draw (0:\R)
   \foreach \x in {30,60,90,...,360} {  -- (\x:\R) };
   \draw[thick, color=purple] (240: \R) -- (270:\R);
   \node at (270:\R+3) {\tiny{$\infty$}};
   \node at (240:\R+4) {\tiny{$\zero$}};
   \node at (210:\R+4) {\tiny{$a$}};
   \node at (180:\R+6) {\tiny{$\infty$}};
   \node at (150:\R+4) {\tiny{$a$}};
   \node at (120:\R+4) {\tiny{$b$}};
   \node at (90:\R+4) {\tiny{$\zero$}};
   \node at (60:\R+6) {\tiny{$\infty$}};
   \node at (30:\R+4) {\tiny{$b$}};
   \node at (360:\R+4) {\tiny{$a$}};
   \node at (330:\R+4) {\color{red}\tiny{$\zero$}};
   \node at (300:\R+4) {\tiny{$b$}};
    \draw(210:\R)--(270:\R);
    \draw(90:\R) -- (270:\R);
    \draw(90:\R)--(210:\R);
    \draw(90:\R)--(180:\R);
    \draw(90:\R)--(150:\R);
    \draw(30:\R)--(90:\R);
    \draw(30:\R)--(270:\R);
    \draw(330:\R)--(30:\R);
    \draw(330:\R)--(270:\R);
    \end{scope}
    \end{scope}
\end{scope}

\begin{scope}[shift={(0, -4)}]
   \draw (0:\R)
   \foreach \x in {30,60,90,...,360} {  -- (\x:\R) };
   \draw[thick, color=purple] (240: \R) -- (270:\R);
   \node at (270:\R+3) {\tiny{$\infty$}};
   \node at (240:\R+4) {\tiny{$\zero$}};
   \node at (210:\R+4) {\color{red}\tiny{$a$}};
   \node at (180:\R+6) {\tiny{$\infty$}};
   \node at (150:\R+4) {\tiny{$a$}};
   \node at (120:\R+4) {\tiny{$b$}};
   \node at (90:\R+4) {\tiny{$\zero$}};
   \node at (60:\R+6) {\tiny{$\infty$}};
   \node at (30:\R+4) {\tiny{$\zero$}};
   \node at (360:\R+4) {\tiny{$a$}};
   \node at (330:\R+4) {\tiny{$b$}};
   \node at (300:\R+4) {\tiny{$a$}};
    \draw(210:\R)--(270:\R);

\draw[->] (2,0) -- (3,0);

\begin{scope}[shift={(5,0)}]

   \draw (0:\R)
   \foreach \x in {30,60,90,...,360} {  -- (\x:\R) };
   \draw[thick, color=purple] (240: \R) -- (270:\R);
   \node at (270:\R+3) {\tiny{$\infty$}};
   \node at (240:\R+4) {\tiny{$\zero$}};
   \node at (210:\R+4) {\tiny{$a$}};
   \node at (180:\R+6) {\tiny{$\infty$}};
   \node at (150:\R+4) {\tiny{$a$}};
   \node at (120:\R+4) {\tiny{$b$}};
   \node at (90:\R+4) {\color{red}\tiny{$\zero$}};
   \node at (60:\R+6) {\tiny{$\infty$}};
   \node at (30:\R+4) {\tiny{$\zero$}};
   \node at (360:\R+4) {\tiny{$a$}};
   \node at (330:\R+4) {\tiny{$b$}};
   \node at (300:\R+4) {\tiny{$a$}};
    \draw(210:\R)--(270:\R);
    \draw(90:\R) -- (270:\R);
    \draw(90:\R)--(210:\R);
    \draw(90:\R)--(180:\R);
    \draw(90:\R)--(150:\R);
    \draw[->] (2,0) -- (3,0);

    \begin{scope}[shift={(5,0)}]

   \draw (0:\R)
   \foreach \x in {30,60,90,...,360} {  -- (\x:\R) };
   \draw[thick, color=purple] (240: \R) -- (270:\R);
   \node at (270:\R+3) {\tiny{$\infty$}};
   \node at (240:\R+4) {\tiny{$\zero$}};
   \node at (210:\R+4) {\tiny{$a$}};
   \node at (180:\R+6) {\tiny{$\infty$}};
   \node at (150:\R+4) {\tiny{$a$}};
   \node at (120:\R+4) {\tiny{$b$}};
   \node at (90:\R+4) {\tiny{$\zero$}};
   \node at (60:\R+6) {\tiny{$\infty$}};
   \node at (30:\R+4) {\tiny{$\zero$}};
   \node at (360:\R+4) {\color{red}\tiny{$a$}};
   \node at (330:\R+4) {\tiny{$b$}};
   \node at (300:\R+4) {\tiny{$a$}};
    \draw(210:\R)--(270:\R);
    \draw(90:\R) -- (270:\R);
    \draw(90:\R)--(210:\R);
    \draw(90:\R)--(180:\R);
    \draw(90:\R)--(150:\R);
    \draw(360:\R)--(90:\R);
    \draw(360:\R)--(60:\R);
    \draw(360:\R)--(270:\R);
 
\draw[->] (2,0) -- (3,0);
\node at (2.5, 0.5) {!};
    \begin{scope}[shift={(5,0)}]

   \draw (0:\R)
   \foreach \x in {30,60,90,...,360} {  -- (\x:\R) };
   \draw[thick, color=purple] (240: \R) -- (270:\R);
   \node at (270:\R+3) {\tiny{$\infty$}};
   \node at (240:\R+4) {\tiny{$\zero$}};
   \node at (210:\R+4) {\tiny{$a$}};
   \node at (180:\R+6) {\tiny{$\infty$}};
   \node at (150:\R+4) {\tiny{$a$}};
   \node at (120:\R+4) {\tiny{$b$}};
   \node at (90:\R+4) {\tiny{$\zero$}};
   \node at (60:\R+6) {\tiny{$\infty$}};
   \node at (30:\R+4) {\tiny{$\zero$}};
   \node at (360:\R+4) {\tiny{$a$}};
   \node at (330:\R+4) {\tiny{$b$}};
   \node at (300:\R+4) {\tiny{$a$}};
    \draw(210:\R)--(270:\R);
    \draw(90:\R) -- (270:\R);
    \draw(90:\R)--(210:\R);
    \draw(90:\R)--(180:\R);
    \draw(90:\R)--(150:\R);
    \draw(30:\R)--(300:\R);
    \draw(30:\R)--(330:\R);
    \draw(300:\R)--(90:\R);
    \draw(300:\R)--(60:\R);
    \end{scope}
    \end{scope}
\end{scope}
\begin{scope}[shift={(0, -4)}]
   \draw (0:\R)
   \foreach \x in {30,60,90,...,360} {  -- (\x:\R) };
   \draw[thick, color=purple] (240: \R) -- (270:\R);
   \node at (270:\R+3) {\tiny{$\infty$}};
   \node at (240:\R+4) {\tiny{$\zero$}};
   \node at (210:\R+4) {\color{red}\tiny{$a$}};
   \node at (180:\R+6) {\tiny{$\infty$}};
   \node at (150:\R+4) {\tiny{$a$}};
   \node at (120:\R+4) {\tiny{$b$}};
   \node at (90:\R+4) {\tiny{$a$}};
   \node at (60:\R+6) {\tiny{$\zero$}};
   \node at (30:\R+4) {\tiny{$\infty$}};
   \node at (360:\R+4) {\tiny{$\zero$}};
   \node at (330:\R+4) {\tiny{$\infty$}};
   \node at (300:\R+4) {\tiny{$\zero$}};
    \draw(210:\R)--(270:\R);

\draw[->] (2,0) -- (3,0);

\begin{scope}[shift={(5,0)}]

   \draw (0:\R)
   \foreach \x in {30,60,90,...,360} {  -- (\x:\R) };
   \draw[thick, color=purple] (240: \R) -- (270:\R);
   \node at (270:\R+3) {\tiny{$\infty$}};
   \node at (240:\R+4) {\tiny{$\zero$}};
   \node at (210:\R+4) {\tiny{$a$}};
   \node at (180:\R+6) {\tiny{$\infty$}};
   \node at (150:\R+4) {\tiny{$a$}};
   \node at (120:\R+4) {\tiny{$b$}};
   \node at (90:\R+4) {\tiny $a$};
   \node at (60:\R+6) {\color{red}\tiny{$\zero$}};
   \node at (30:\R+4) {\tiny{$\infty$}};
   \node at (360:\R+4) {\tiny{$\zero$}};
   \node at (330:\R+4) {\tiny{$\infty$}};
   \node at (300:\R+4) {\tiny{$\zero$}};
    \draw(210:\R)--(270:\R);
    \draw(60:\R) -- (270:\R);
    \draw(60:\R)--(210:\R);
    \draw(60:\R)--(180:\R);
    \draw(60:\R)--(150:\R);
    \draw(60:\R)--(120:\R);
    \draw[->] (2,0) -- (3,0);
    \node at (2.5, 0.5) {!};

    \begin{scope}[shift={(5,0)}]

   \draw (0:\R)
   \foreach \x in {30,60,90,...,360} {  -- (\x:\R) };
   \draw[thick, color=purple] (240: \R) -- (270:\R);
   \node at (270:\R+3) {\tiny{$\infty$}};
   \node at (240:\R+4) {\tiny{$\zero$}};
   \node at (210:\R+4) {\tiny{$a$}};
   \node at (180:\R+6) {\tiny{$\infty$}};
   \node at (150:\R+4) {\tiny{$a$}};
   \node at (120:\R+4) {\tiny{$b$}};
   \node at (90:\R+4) {\tiny{$a$}};
   \node at (60:\R+6) {\tiny{$\zero$}};
   \node at (30:\R+4) {\tiny{$\infty$}};
   \node at (360:\R+4) {\tiny $\zero$};
   \node at (330:\R+4) {\tiny{$\infty$}};
   \node at (300:\R+4) {\tiny{$\zero$}};
    \draw(210:\R)--(270:\R);
    \draw(60:\R)--(210:\R);
    \draw(60:\R)--(180:\R);
    \draw(60:\R)--(150:\R);
    \draw(60:\R)--(120:\R);
    \draw(210:\R)--(30:\R);
    \draw(210:\R)--(360:\R);
    \draw(210:\R)--(330:\R);
    \draw(210:\R)--(300:\R);
    \end{scope}
\end{scope}
\end{scope}
\end{scope}
\end{tikzpicture}
\end{center}
\caption{Examples of Algorithm A that inputs a point $y$ in $X'(m)$ and outputs a triangulation of $T$ of the $m+1$-gon, in such a way that $y$ is in the cluster torus defined by $T$. Here, $a\neq b$ and neither are $\zero, \infty$, In the top example, we do not have to backtrack. In the middle example, we backtrack at the ! symbol, since there is no $\zero$ after the red ${\color{red} a}$. In the bottom example, we backtrack at the !  symbol, since there is no element different from $\zero, \infty$ after the red ${\color{red} \zero}$.}
\label{fig:examples-algorithm}
\end{figure}


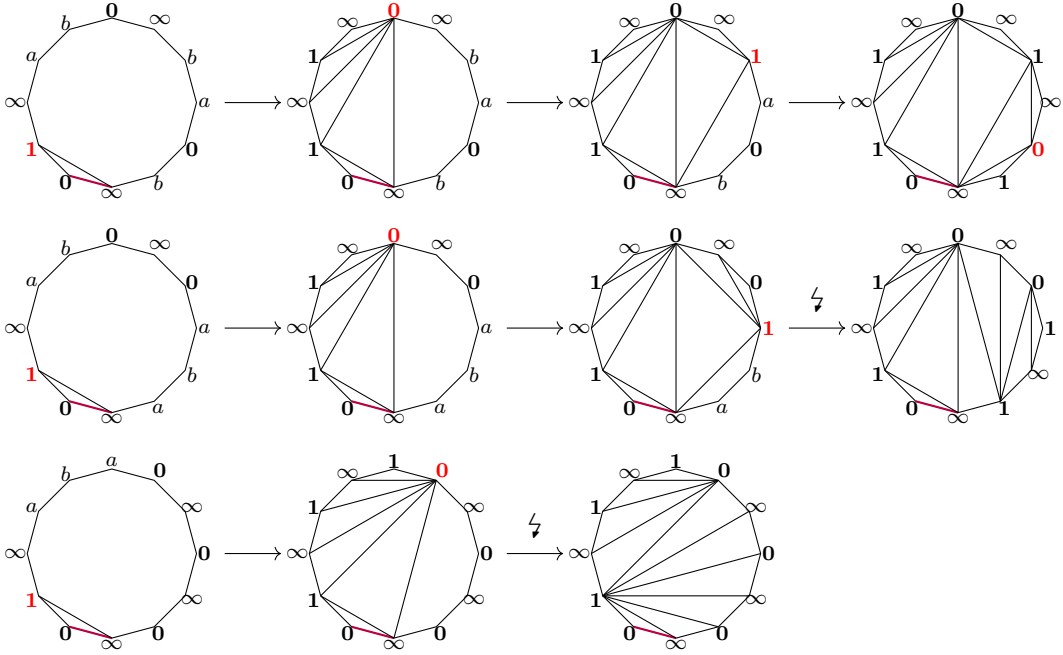
\begin{figure}
    \begin{center}
\begin{tikzpicture}[scale=0.7]
   \newdimen\R
\R=1.5cm
   \draw (0:\R)
   \foreach \x in {30,60,90,...,360} {  -- (\x:\R) };
   \draw[thick, color=purple] (240: \R) -- (270:\R);
   \node at (270:\R+3) {\tiny{$\infty$}};
   \node at (240:\R+4) {\tiny{$\zero$}};
   \node at (210:\R+4) {\color{red}\tiny{$\one$}};
   \node at (180:\R+6) {\tiny{$\infty$}};
   \node at (150:\R+4) {\tiny{$a$}};
   \node at (120:\R+4) {\tiny{$b$}};
   \node at (90:\R+4) {\tiny{$\zero$}};
   \node at (60:\R+6) {\tiny{$\infty$}};
   \node at (30:\R+4) {\tiny{$b$}};
   \node at (360:\R+4) {\tiny{$a$}};
   \node at (330:\R+4) {\tiny{$\zero$}};
   \node at (300:\R+4) {\tiny{$b$}};
    \draw(210:\R)--(270:\R);

\draw[->] (2,0) -- (3,0);

\begin{scope}[shift={(5,0)}]

   \draw (0:\R)
   \foreach \x in {30,60,90,...,360} {  -- (\x:\R) };
   \draw[thick, color=purple] (240: \R) -- (270:\R);
   \node at (270:\R+3) {\tiny{$\infty$}};
   \node at (240:\R+4) {\tiny{$\zero$}};
   \node at (210:\R+4) {\tiny{$\one$}};
   \node at (180:\R+6) {\tiny{$\infty$}};
   \node at (150:\R+4) {\tiny{$\one$}};
   \node at (120:\R+4) {\tiny{$\infty$}};
   \node at (90:\R+4) {\color{red}\tiny{$\zero$}};
   \node at (60:\R+6) {\tiny{$\infty$}};
   \node at (30:\R+4) {\tiny{$b$}};
   \node at (360:\R+4) {\tiny{$a$}};
   \node at (330:\R+4) {\tiny{$\zero$}};
   \node at (300:\R+4) {\tiny{$b$}};
    \draw(210:\R)--(270:\R);
    \draw(90:\R) -- (270:\R);
    \draw(90:\R)--(210:\R);
    \draw(90:\R)--(180:\R);
    \draw(90:\R)--(150:\R);
    \draw[->] (2,0) -- (3,0);

    \begin{scope}[shift={(5,0)}]

   \draw (0:\R)
   \foreach \x in {30,60,90,...,360} {  -- (\x:\R) };
   \draw[thick, color=purple] (240: \R) -- (270:\R);
  \node at (270:\R+3) {\tiny{$\infty$}};
   \node at (240:\R+4) {\tiny{$\zero$}};
   \node at (210:\R+4) {\tiny{$\one$}};
   \node at (180:\R+6) {\tiny{$\infty$}};
   \node at (150:\R+4) {\tiny{$\one$}};
   \node at (120:\R+4) {\tiny{$\infty$}};
   \node at (90:\R+4) {\tiny{$\zero$}};
   \node at (60:\R+6) {\tiny{$\infty$}};
   \node at (30:\R+4) {\color{red}\tiny{$\one$}};
   \node at (360:\R+4) {\tiny{$a$}};
   \node at (330:\R+4) {\tiny{$\zero$}};
   \node at (300:\R+4) {\tiny{$b$}};
    \draw(210:\R)--(270:\R);
    \draw(90:\R) -- (270:\R);
    \draw(90:\R)--(210:\R);
    \draw(90:\R)--(180:\R);
    \draw(90:\R)--(150:\R);
    \draw(30:\R)--(90:\R);
    \draw(30:\R)--(270:\R);
\draw[->] (2,0) -- (3,0);
    \begin{scope}[shift={(5,0)}]

   \draw (0:\R)
   \foreach \x in {30,60,90,...,360} {  -- (\x:\R) };
   \draw[thick, color=purple] (240: \R) -- (270:\R);
  \node at (270:\R+3) {\tiny{$\infty$}};
   \node at (240:\R+4) {\tiny{$\zero$}};
   \node at (210:\R+4) {\tiny{$\one$}};
   \node at (180:\R+6) {\tiny{$\infty$}};
   \node at (150:\R+4) {\tiny{$\one$}};
   \node at (120:\R+4) {\tiny{$\infty$}};
   \node at (90:\R+4) {\tiny{$\zero$}};
   \node at (60:\R+6) {\tiny{$\infty$}};
   \node at (30:\R+4) {\tiny{$\one$}};
   \node at (360:\R+4) {\tiny{$\infty$}};
   \node at (330:\R+4) {\color{red}\tiny{$\zero$}};
   \node at (300:\R+4) {\tiny{$\one$}};
    \draw(210:\R)--(270:\R);
    \draw(90:\R) -- (270:\R);
    \draw(90:\R)--(210:\R);
    \draw(90:\R)--(180:\R);
    \draw(90:\R)--(150:\R);
    \draw(30:\R)--(90:\R);
    \draw(30:\R)--(270:\R);
    \draw(330:\R)--(30:\R);
    \draw(330:\R)--(270:\R);
    \end{scope}
    \end{scope}
\end{scope}

\begin{scope}[shift={(0, -4)}]
   \draw (0:\R)
   \foreach \x in {30,60,90,...,360} {  -- (\x:\R) };
   \draw[thick, color=purple] (240: \R) -- (270:\R);
   \node at (270:\R+3) {\tiny{$\infty$}};
   \node at (240:\R+4) {\tiny{$\zero$}};
   \node at (210:\R+4) {\color{red}\tiny{$\one$}};
   \node at (180:\R+6) {\tiny{$\infty$}};
   \node at (150:\R+4) {\tiny{$a$}};
   \node at (120:\R+4) {\tiny{$b$}};
   \node at (90:\R+4) {\tiny{$\zero$}};
   \node at (60:\R+6) {\tiny{$\infty$}};
   \node at (30:\R+4) {\tiny{$\zero$}};
   \node at (360:\R+4) {\tiny{$a$}};
   \node at (330:\R+4) {\tiny{$b$}};
   \node at (300:\R+4) {\tiny{$a$}};
    \draw(210:\R)--(270:\R);

\draw[->] (2,0) -- (3,0);

\begin{scope}[shift={(5,0)}]

   \draw (0:\R)
   \foreach \x in {30,60,90,...,360} {  -- (\x:\R) };
   \draw[thick, color=purple] (240: \R) -- (270:\R);
   \node at (270:\R+3) {\tiny{$\infty$}};
   \node at (240:\R+4) {\tiny{$\zero$}};
   \node at (210:\R+4) {\tiny{$\one$}};
   \node at (180:\R+6) {\tiny{$\infty$}};
   \node at (150:\R+4) {\tiny{$\one$}};
   \node at (120:\R+4) {\tiny{$\infty$}};
   \node at (90:\R+4) {\color{red}\tiny{$\zero$}};
   \node at (60:\R+6) {\tiny{$\infty$}};
   \node at (30:\R+4) {\tiny{$\zero$}};
   \node at (360:\R+4) {\tiny{$a$}};
   \node at (330:\R+4) {\tiny{$b$}};
   \node at (300:\R+4) {\tiny{$a$}};
    \draw(210:\R)--(270:\R);
    \draw(90:\R) -- (270:\R);
    \draw(90:\R)--(210:\R);
    \draw(90:\R)--(180:\R);
    \draw(90:\R)--(150:\R);
    \draw[->] (2,0) -- (3,0);

    \begin{scope}[shift={(5,0)}]

   \draw (0:\R)
   \foreach \x in {30,60,90,...,360} {  -- (\x:\R) };
   \draw[thick, color=purple] (240: \R) -- (270:\R);
   \node at (270:\R+3) {\tiny{$\infty$}};
   \node at (240:\R+4) {\tiny{$\zero$}};
   \node at (210:\R+4) {\tiny{$\one$}};
   \node at (180:\R+6) {\tiny{$\infty$}};
   \node at (150:\R+4) {\tiny{$\one$}};
   \node at (120:\R+4) {\tiny{$\infty$}};
   \node at (90:\R+4) {\tiny{$\zero$}};
   \node at (60:\R+6) {\tiny{$\infty$}};
   \node at (30:\R+4) {\tiny{$\zero$}};
   \node at (360:\R+4) {\color{red}\tiny{$\one$}};
   \node at (330:\R+4) {\tiny{$b$}};
   \node at (300:\R+4) {\tiny{$a$}};
    \draw(210:\R)--(270:\R);
    \draw(90:\R) -- (270:\R);
    \draw(90:\R)--(210:\R);
    \draw(90:\R)--(180:\R);
    \draw(90:\R)--(150:\R);
    \draw(360:\R)--(90:\R);
    \draw(360:\R)--(60:\R);
    \draw(360:\R)--(270:\R);
 
\draw[->] (2,0) -- (3,0);
\node at (2.5, 0.5) {!};
    \begin{scope}[shift={(5,0)}]

   \draw (0:\R)
   \foreach \x in {30,60,90,...,360} {  -- (\x:\R) };
   \draw[thick, color=purple] (240: \R) -- (270:\R);
   \node at (270:\R+3) {\tiny{$\infty$}};
   \node at (240:\R+4) {\tiny{$\zero$}};
   \node at (210:\R+4) {\tiny{$\one$}};
   \node at (180:\R+6) {\tiny{$\infty$}};
   \node at (150:\R+4) {\tiny{$\one$}};
   \node at (120:\R+4) {\tiny{$\infty$}};
   \node at (90:\R+4) {\tiny{$\zero$}};
   \node at (60:\R+6) {\tiny{$\infty$}};
   \node at (30:\R+4) {\tiny{$\zero$}};
   \node at (360:\R+4) {\tiny{$\one$}};
   \node at (330:\R+4) {\tiny{$\infty$}};
   \node at (300:\R+4) {\tiny{$\one$}};
    \draw(210:\R)--(270:\R);
    \draw(90:\R) -- (270:\R);
    \draw(90:\R)--(210:\R);
    \draw(90:\R)--(180:\R);
    \draw(90:\R)--(150:\R);
    \draw(30:\R)--(300:\R);
    \draw(30:\R)--(330:\R);
    \draw(300:\R)--(90:\R);
    \draw(300:\R)--(60:\R);
    \end{scope}
    \end{scope}
\end{scope}
\begin{scope}[shift={(0, -4)}]
   \draw (0:\R)
   \foreach \x in {30,60,90,...,360} {  -- (\x:\R) };
   \draw[thick, color=purple] (240: \R) -- (270:\R);
   \node at (270:\R+3) {\tiny{$\infty$}};
   \node at (240:\R+4) {\tiny{$\zero$}};
   \node at (210:\R+4) {\color{red}\tiny{$\one$}};
   \node at (180:\R+6) {\tiny{$\infty$}};
   \node at (150:\R+4) {\tiny{$a$}};
   \node at (120:\R+4) {\tiny{$b$}};
   \node at (90:\R+4) {\tiny{$a$}};
   \node at (60:\R+6) {\tiny{$\zero$}};
   \node at (30:\R+4) {\tiny{$\infty$}};
   \node at (360:\R+4) {\tiny{$\zero$}};
   \node at (330:\R+4) {\tiny{$\infty$}};
   \node at (300:\R+4) {\tiny{$\zero$}};
    \draw(210:\R)--(270:\R);

\draw[->] (2,0) -- (3,0);

\begin{scope}[shift={(5,0)}]

   \draw (0:\R)
   \foreach \x in {30,60,90,...,360} {  -- (\x:\R) };
   \draw[thick, color=purple] (240: \R) -- (270:\R);
   \node at (270:\R+3) {\tiny{$\infty$}};
   \node at (240:\R+4) {\tiny{$\zero$}};
   \node at (210:\R+4) {\tiny{$\one$}};
   \node at (180:\R+6) {\tiny{$\infty$}};
   \node at (150:\R+4) {\tiny{$\one$}};
   \node at (120:\R+4) {\tiny{$\infty$}};
   \node at (90:\R+4) {\tiny $\one$};
   \node at (60:\R+6) {\color{red}\tiny{$\zero$}};
   \node at (30:\R+4) {\tiny{$\infty$}};
   \node at (360:\R+4) {\tiny{$\zero$}};
   \node at (330:\R+4) {\tiny{$\infty$}};
   \node at (300:\R+4) {\tiny{$\zero$}};
    \draw(210:\R)--(270:\R);
    \draw(60:\R) -- (270:\R);
    \draw(60:\R)--(210:\R);
    \draw(60:\R)--(180:\R);
    \draw(60:\R)--(150:\R);
    \draw(60:\R)--(120:\R);
    \draw[->] (2,0) -- (3,0);
    \node at (2.5, 0.5) {!};

    \begin{scope}[shift={(5,0)}]

   \draw (0:\R)
   \foreach \x in {30,60,90,...,360} {  -- (\x:\R) };
   \draw[thick, color=purple] (240: \R) -- (270:\R);
   \node at (270:\R+3) {\tiny{$\infty$}};
   \node at (240:\R+4) {\tiny{$\zero$}};
   \node at (210:\R+4) {\tiny{$\one$}};
   \node at (180:\R+6) {\tiny{$\infty$}};
   \node at (150:\R+4) {\tiny{$\one$}};
   \node at (120:\R+4) {\tiny{$\infty$}};
   \node at (90:\R+4) {\tiny{$\one$}};
   \node at (60:\R+6) {\tiny{$\zero$}};
   \node at (30:\R+4) {\tiny{$\infty$}};
   \node at (360:\R+4) {\tiny $\zero$};
   \node at (330:\R+4) {\tiny{$\infty$}};
   \node at (300:\R+4) {\tiny{$\zero$}};
    \draw(210:\R)--(270:\R);
    \draw(60:\R)--(210:\R);
    \draw(60:\R)--(180:\R);
    \draw(60:\R)--(150:\R);
    \draw(60:\R)--(120:\R);
    \draw(210:\R)--(30:\R);
    \draw(210:\R)--(360:\R);
    \draw(210:\R)--(330:\R);
    \draw(210:\R)--(300:\R);
    \end{scope}
\end{scope}
\end{scope}
\end{scope}
\end{tikzpicture}
\end{center}
\caption{Examples of the procedure that inputs a point $y$ in $X'(m)$ and outputs an element $z \in X'_{\F_2}(m)$, compare with Figure \ref{fig:examples-algorithm}.}
\label{fig:examples-algorithm-B}
\end{figure}
\subsection{Counterexample}\label{sec:counterexample}

In this section, we verify the existence of an $\F_2$-covering that is not an $\F$-covering for any field $\F \not\cong \F_2$. 

Recall that if $y \in X_{\F}(m)$ then a diagonal $ij$ of the polygon $P_m$ is invalid for $y$ if there does not exist a triangulation $T$ of $P_{m+1}$ containing $ij$ such that $y$ is in the corresponding cluster torus, cf. Lemma \ref{lem:invalid-diagonal}. Let us denote the set of invalid diagonals of $y$ by $I(y)$. 

\begin{lemma}\label{lem:invalid-diagonals}
Let $z \in X_{\F_2}(m)$ and $y \in X_{\F}(m)$. The following conditions are equivalent.
\begin{enumerate}
    \item[(1)] There exists a triangulation $T$ of $P_{m+1}$ that admits $z$ as a proper coloring, but not $y$.
    \item[(2)] $I(y) \not\subseteq I(z)$.
\end{enumerate}
\end{lemma}
\begin{proof}
    (1) $\Rightarrow$ (2). Let $T$ be a triangulation as in  (1), and note that every diagonal in this triangulation is valid for $z$. If every diagonal were also valid for $y$, then $y$ would be a proper coloring of $T$. Thus, there must be a diagonal that is invalid for $y$ but valid for $z$, i.e., $I(y) \not\subseteq I(z)$.

    (2) $\Rightarrow$ (1). Let $ij \in I(y) \setminus I(z)$. Since $ij$ is valid for $z$, there exists a triangulation having $ij$ as a diagonal that admits $z$ as a proper coloring. But this triangulation cannot admit $y$ as a proper coloring, since $(ij) \in I(y)$.
\end{proof}

We will use Lemma \ref{lem:invalid-diagonals} as follows.

\begin{lemma}\label{lem:key-counterexample}
Assume there exists $y \in X'_{\F}(m)$ such that for every $z \in X'_{\F_2}(m)$, we have $I(y) \not\subseteq I(z)$. Then, there exists a $\F_2$-covering collection that is not $\F$-covering.
\end{lemma}
\begin{proof}
By the assumptions of the lemma, for every $z \in X'_{\F_2}(m)$ there exists a triangulation $T_x$ such that $y$ does not belong to the cluster torus defined by $T_z$. The collection $\{T_z : z \in X'_{\F_2}(m)\}$ is then the desired collection, since it does not cover $y$. 
\end{proof}

Now let $m = 11$, so we consider triangulations of a $12$-gon, and let $\F \not\cong \F_2$ be any field. Then, $\# \F\PP^1 \geq 4$, so we can find distinct $a, b \in \F\PP^1 \setminus\{0, \infty\}$.  Let $y = (0,a,\infty,a,b,0,\infty,b,a,0,b,\infty)$. We will show that $y$ satisfies the conditions of Lemma \ref{lem:key-counterexample}, so by the proof of this lemma, this will give a collection that is $\F_2$-covering but not $\F$-covering. Since $\F$ is arbitrary, this will be the collection we are looking for.

We depict the element $y$, together with its invalid diagonals, in Figure \ref{fig:ex1}.

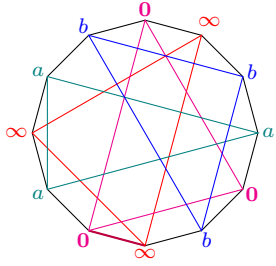
\begin{figure}
\begin{tikzpicture}
\node at (-7, 0) {};
      \newdimen\R
\R=1.5cm
   \draw (0:\R)
   \foreach \x in {30,60,90,...,360} {  -- (\x:\R) };
   \draw[thick, color=purple] (240: \R) -- (270:\R);
   \node at (270:\R+3) {\color{red}\tiny{$\infty$}};
   \node at (240:\R+4) {\color{magenta}\tiny{$\zero$}};
   \node at (210:\R+4) {\color{teal}\tiny{$a$}};
   \node at (180:\R+6) {\color{red}\tiny{$\infty$}};
   \node at (150:\R+4) {\color{teal}\tiny{$a$}};
   \node at (120:\R+4) {\color{blue} \tiny{$b$}};
   \node at (90:\R+4) {\color{magenta} \tiny{$\zero$}};
   \node at (60:\R+6) {\color{red}\tiny{$\infty$}};
   \node at (30:\R+4) {\color{blue} \tiny{$b$}};
   \node at (360:\R+4) {\color{teal}\tiny{$a$}};
   \node at (330:\R+4) {\color{magenta} \tiny{$\zero$}};
   \node at (300:\R+4) {\color{blue} \tiny{$b$}};
   \draw[color=red] (270:\R)--(180:\R)--(60:\R)--cycle;
   \draw[color=magenta] (240:\R)--(90:\R)--(330:\R)--cycle;
   \draw[color=teal] (210:\R)--(150:\R)--(360:\R)--cycle;
   \draw[color=blue] (120:\R)--(30:\R)--(300:\R)--cycle;
\end{tikzpicture}
\caption{The element $y = (\zero, a, \infty, a, b, \zero, \infty, b, a, \zero, b, \infty)$ and its invalid diagonals. Note that a diagonal is invalid if and only if it joins two of the same labels, and we have color-coded the diagonals depending on the labels they join.}
\label{fig:ex1}
\end{figure}

We assert that for every $z \in X'_{\F_2}(m)$ we have that $I(y) \not \subseteq I(z)$. To prove this we proceed by contradiction. Suppose that there exists $z \in X'_{\F_2}(m)$ with $I(y) \subseteq I(z)$, so every invalid edge of $y$ is also an invalid edge of $z$. We use this to determine the points that $z = (z_0,\dots,z_{11})$ can be.

\underline{\textbf{Claim 1:}} $z_6 = \infty$. To verify this, suppose that $z_6 \in \{\zero,\one\}$. This gives us two cases:

\begin{enumerate}
    \item $z_6 = \zero$. In this case the diagonal $(6,11)$ is a diagonal that is invalid for $z$ and between two vertices of two different colors. So one of the two polygons in which this diagonal divides the $12$-gon is such that there is no triangulation that has $z$ as a proper coloring. It follows that $z_i \in \{\zero,\infty\}$ for every $6 \leq i \leq 10$, or for every $1 \leq i \leq 6$. Since, further, we must have $z_j = \one$ for some $j$,  we claim that this will make one of the invalid diagonals for $y$, valid for $z$, which is a contradiction. To prove our claim, from Figure $\ref{fig:ex2}$ we see that, regardless of the value of $j$ that makes $z_j = \one$, we can find a diagonal $(j,k)$ crossing the diagonal $(6,11)$ that is invalid for $y$ and such that $z$ takes all values $\zero, \one, \infty$ on both sides of the diagonal $(j,k)$. The latter condition implies that $(j,k)$ is valid for $z$, which was our claim. 
        \begin{figure}
\begin{tikzpicture}
\node at (-3,0) {};
      \newdimen\R
\R=1.5cm
   \draw (0:\R)
   \foreach \x in {30,60,90,...,360} {  -- (\x:\R) };
   \draw[thick, color=purple] (240: \R) -- (270:\R);
   \node at (270:\R+3) {\tiny{$\infty$}};
   \node at (240:\R+4) {\tiny{$\zero$}};
   \node at (210:\R+4) {\tiny{$z_1$}};
   \node at (180:\R+6) {\tiny{$z_2$}};
   \node at (150:\R+4) {\tiny{$z_3$}};
   \node at (120:\R+4) {\tiny{$z_4$}};
   \node at (90:\R+4) {\tiny{$z_5$}};
   \node at (60:\R+6) {\tiny{$\zero$}};
   \node at (30:\R+4) {\tiny{$z_7$}};
   \node at (360:\R+4) {\tiny{$z_8$}};
   \node at (330:\R+4) {\tiny{$z_9$}};
   \node at (300:\R+4) {\tiny{$z_{10}$}};
   \draw[color=red] (270:\R)--(180:\R)--(60:\R)--cycle;
   \draw[color=magenta] (240:\R)--(90:\R)--(330:\R)--cycle;
   \draw[color=teal] (210:\R)--(150:\R)--(360:\R)--cycle;
   \draw[color=blue] (120:\R)--(30:\R)--(300:\R)--cycle;

   \node at (2.5,0) {$\Rightarrow$};

    \begin{scope}[shift = {(5,0)}]
   \draw (0:\R)
   \foreach \x in {30,60,90,...,360} {  -- (\x:\R) };
   \draw[thick, color=purple] (240: \R) -- (270:\R);
   \node at (270:\R+3) {\tiny{$\infty$}};
   \node at (240:\R+4) {\tiny{$\zero$}};
   \node at (210:\R+4) {\tiny{$z_1$}};
   \node at (180:\R+6) {\tiny{$z_2$}};
   \node at (150:\R+4) {\tiny{$z_3$}};
   \node at (120:\R+4) {\tiny{$z_4$}};
   \node at (90:\R+4) {\tiny{$z_5$}};
   \node at (60:\R+6) {\tiny{$\zero$}};
   \node at (30:\R+4) {\tiny{$\infty$}};
   \node at (360:\R+4) {\tiny{$\zero$}};
   \node at (330:\R+4) {\tiny{$\infty$}};
   \node at (300:\R+4) {\tiny{$\zero$}};
   \draw[color=red] (270:\R)--(180:\R)--(60:\R)--cycle;
   \draw[color=magenta] (240:\R)--(90:\R)--(330:\R)--cycle;
   \draw[color=teal] (210:\R)--(150:\R)--(360:\R)--cycle;
   \draw[color=blue] (120:\R)--(30:\R)--(300:\R)--cycle;

   \node at (2.5,0) {or};

    \begin{scope}[shift = {(5,0)}]
   \draw (0:\R)
   \foreach \x in {30,60,90,...,360} {  -- (\x:\R) };
   \draw[thick, color=purple] (240: \R) -- (270:\R);
   \node at (270:\R+3) {\tiny{$\infty$}};
   \node at (240:\R+4) {\tiny{$\zero$}};
   \node at (210:\R+4) {\tiny{$\infty$}};
   \node at (180:\R+6) {\tiny{$\zero$}};
   \node at (150:\R+4) {\tiny{$\infty$}};
   \node at (120:\R+4) {\tiny{$\zero$}};
   \node at (90:\R+4) {\tiny{$\infty$}};
   \node at (60:\R+6) {\tiny{$\zero$}};
   \node at (30:\R+4) {\tiny{$z_7$}};
   \node at (360:\R+4) {\tiny{$z_8$}};
   \node at (330:\R+4) {\tiny{$z_9$}};
   \node at (300:\R+4) {\tiny{$z_{10}$}};
   \draw[color=red] (270:\R)--(180:\R)--(60:\R)--cycle;
   \draw[color=magenta] (240:\R)--(90:\R)--(330:\R)--cycle;
   \draw[color=teal] (210:\R)--(150:\R)--(360:\R)--cycle;
   \draw[color=blue] (120:\R)--(30:\R)--(300:\R)--cycle;
    \end{scope}
    \end{scope}
\end{tikzpicture}
\caption{
Assuming $z_6 = \zero$, since the red diagonal joining $\zero$ with $\infty$ is invalid, we obtain one of the two colorings of the right. Since one of the labels must be $\one$, this will imply that one of the drawn diagonals is actually valid for $z$, a contradiction.}
\label{fig:ex2}
    \end{figure}
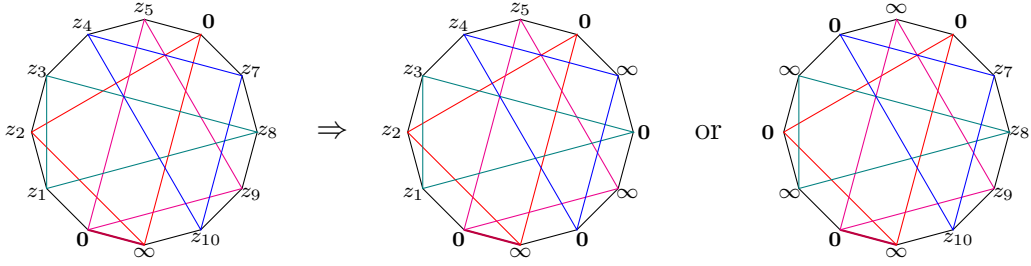

    \item $z_6 = \one$. Once again the diagonal $(6,11)$ is an invalid (for $z$) diagonal drawn between vertices of different colors, so $z_i \in \{\one,\infty\}$ for $6 \leq i \leq 11$, or $z_i \in \{\zero, \one\}$ for $0 \leq i \leq 6$. It is easy to see that the second case is not possible, so we only consider the first case in which $z$ is as shown in Figure \ref{fig:ex8} and we obtain that the diagonal $(0,9)$ is valid for $z$, a contradiction.

    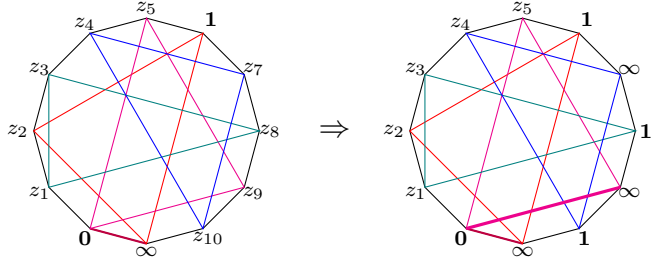
\begin{figure}
\begin{tikzpicture}
\node at (-5, 0) {};
      \newdimen\R
\R=1.5cm
   \draw (0:\R)
   \foreach \x in {30,60,90,...,360} {  -- (\x:\R) };
   \draw[thick, color=purple] (240: \R) -- (270:\R);
   \node at (270:\R+3) {\tiny{$\infty$}};
   \node at (240:\R+4) {\tiny{$\zero$}};
   \node at (210:\R+4) {\tiny{$z_1$}};
   \node at (180:\R+6) {\tiny{$z_2$}};
   \node at (150:\R+4) {\tiny{$z_3$}};
   \node at (120:\R+4) {\tiny{$z_4$}};
   \node at (90:\R+4) {\tiny{$z_5$}};
   \node at (60:\R+6) {\tiny{$\one$}};
   \node at (30:\R+4) {\tiny{$z_7$}};
   \node at (360:\R+4) {\tiny{$z_8$}};
   \node at (330:\R+4) {\tiny{$z_9$}};
   \node at (300:\R+4) {\tiny{$z_{10}$}};
   \draw[color=red] (270:\R)--(180:\R)--(60:\R)--cycle;
   \draw[color=magenta] (240:\R)--(90:\R)--(330:\R)--cycle;
   \draw[color=teal] (210:\R)--(150:\R)--(360:\R)--cycle;
   \draw[color=blue] (120:\R)--(30:\R)--(300:\R)--cycle;

   \node at (2.5,0) {$\Rightarrow$};

    \begin{scope}[shift = {(5,0)}]
   \draw (0:\R)
   \foreach \x in {30,60,90,...,360} {  -- (\x:\R) };
   \draw[thick, color=purple] (240: \R) -- (270:\R);
   \node at (270:\R+3) {\tiny{$\infty$}};
   \node at (240:\R+4) {\tiny{$\zero$}};
   \node at (210:\R+4) {\tiny{$z_1$}};
   \node at (180:\R+6) {\tiny{$z_2$}};
   \node at (150:\R+4) {\tiny{$z_3$}};
   \node at (120:\R+4) {\tiny{$z_4$}};
   \node at (90:\R+4) {\tiny{$z_5$}};
   \node at (60:\R+6) {\tiny{$\one$}};
   \node at (30:\R+4) {\tiny{$\infty$}};
   \node at (360:\R+4) {\tiny{$\one$}};
   \node at (330:\R+4) {\tiny{$\infty$}};
   \node at (300:\R+4) {\tiny{$\one$}};
   \draw[color=red] (270:\R)--(180:\R)--(60:\R)--cycle;
   \draw[color=magenta] (240:\R)--(90:\R)--(330:\R)--cycle;
   \draw[color=teal] (210:\R)--(150:\R)--(360:\R)--cycle;
   \draw[color=blue] (120:\R)--(30:\R)--(300:\R)--cycle;
   \draw[very thick, color=magenta] (240:\R)--(330:\R);
   \end{scope}
       \end{tikzpicture}
       \caption{If we assume $z_6 = \one$, then we must have the configuration on the right. But this implies that the thick diagonal is valid for $z$, a contradiction.}
       \label{fig:ex8}
    \end{figure}
\end{enumerate}

Thus, we conclude that $z_6 = \infty$. 

\vspace{1em}

\underline{\textbf{Claim 2:}} $z_2 = \infty$. We follow a similar strategy to Claim 1, assume that $z_2 \in \{\zero, \one\}$ and arrive to a contradiction. 

\begin{enumerate}
    \item $z_2 = \zero$. Since we have shown that $z_6 = \infty$, in this case the diagonal $(2,6)$ is an invalid diagonal between vertices of different colors. It follows that $z_i \in \{\zero,\infty\}$ for $2 \leq i \leq 6$ or for $6 \leq i \leq 10$. But for any of these to happen we would need consecutive repeated entries, which is impossible by the definition of $X(m)$. 

    \item $z_2 = \one$. In this case, we must have $z_1 = \infty$. So the diagonal $(2,11)$ will be valid for $z$ unless $z_3, z_4, \dots, z_{10}$ alternate between $\one$ and $\infty$. But it is easy to see that this cannot happen, see Figure \ref{fig:ex10}. 
    \end{enumerate}
    
     \begin{figure}
\begin{tikzpicture}
\node at (-8, 0) {};
      \newdimen\R
\R=1.5cm
   \draw (0:\R)
   \foreach \x in {30,60,90,...,360} {  -- (\x:\R) };
   \draw[thick, color=purple] (240: \R) -- (270:\R);
   \node at (270:\R+3) {\tiny{$\infty$}};
   \node at (240:\R+4) {\tiny{$\zero$}};
   \node at (210:\R+4) {\tiny{$z_1$}};
   \node at (180:\R+6) {\tiny{$\one$}};
   \node at (150:\R+4) {\tiny{$z_3$}};
   \node at (120:\R+4) {\tiny{$z_4$}};
   \node at (90:\R+4) {\tiny{$z_5$}};
   \node at (60:\R+6) {\tiny{$\infty$}};
   \node at (30:\R+4) {\tiny{$z_7$}};
   \node at (360:\R+4) {\tiny{$z_8$}};
   \node at (330:\R+4) {\tiny{$z_9$}};
   \node at (300:\R+4) {\tiny{$z_{10}$}};
   \draw[color=red] (270:\R)--(180:\R)--(60:\R)--cycle;
   \draw[color=magenta] (240:\R)--(90:\R)--(330:\R)--cycle;
   \draw[color=teal] (210:\R)--(150:\R)--(360:\R)--cycle;
   \draw[color=blue] (120:\R)--(30:\R)--(300:\R)--cycle;
   \draw[very thick, color=red] (180:\R)--(270:\R);
\end{tikzpicture}
\caption{If $z_2 = \one$ then we must have $z_1 = \infty$. For the thick diagonal to be invalid, we must have that $z_3, \dots, z_{11}$ alternate between $\one$ and $\infty$: But since $z_2 = \one$ we must have $z_3 = \infty, z_4 = \one, z_5 = \infty$, a contradiction with $z_6 = \infty$.}
\label{fig:ex10}
\end{figure}
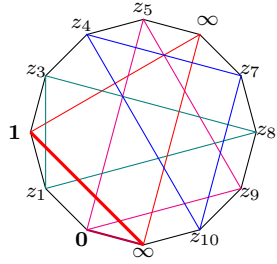

Thus, we conclude that $z_2 = \infty$. Note that it follows that $z_1 = \one$.
The following claims are proved similarly to Claims 1 and 2.

\underline{\textbf{Claim 3:}} $z_3 = \one$. 

\underline{\textbf{Claim 4:}} $z_5 = \zero$. This implies that $z_4 = \infty$. 

\underline{\textbf{Claim 5:}} $z_7 = \zero$.

\underline{\textbf{Claim 6:}} $z_9 = \zero$. This implies that $z_{10} = \one$.

Summarizing, $z$ must have the following form:

\begin{center}
    \begin{tikzpicture}
      \newdimen\R
\R=1.5cm
   \draw (0:\R)
   \foreach \x in {30,60,90,...,360} {  -- (\x:\R) };
   \draw[thick, color=purple] (240: \R) -- (270:\R);
   \node at (270:\R+3) {\tiny{$\infty$}};
   \node at (240:\R+4) {\tiny{$\zero$}};
   \node at (210:\R+4) {\tiny{$\one$}};
   \node at (180:\R+6) {\tiny{$\infty$}};
   \node at (150:\R+4) {\tiny{$\one$}};
   \node at (120:\R+4) {\tiny{$\infty$}};
   \node at (90:\R+4) {\tiny{$\zero$}};
   \node at (60:\R+6) {\tiny{$\infty$}};
   \node at (30:\R+4) {\tiny{$\zero$}};
   \node at (360:\R+4) {\tiny{$z_8$}};
   \node at (330:\R+4) {\tiny{$\zero$}};
   \node at (300:\R+4) {\tiny{$\one$}};
   \draw[color=red] (270:\R)--(180:\R)--(60:\R)--cycle;
   \draw[color=magenta] (240:\R)--(90:\R)--(330:\R)--cycle;
   \draw[color=teal] (210:\R)--(150:\R)--(360:\R)--cycle;
   \draw[color=blue] (120:\R)--(30:\R)--(300:\R)--cycle;
   \draw[very thick, color=blue] (120:\R)--(300:\R);
\end{tikzpicture}
\end{center}

It follows that regardless of the value of $z_8$, the thick diagonal is valid for $z$, which is a contradiction. This proves our claim.

\end{document}